\documentclass[reqno,a4paper,twoside]{amsart}

\usepackage{enumitem}

\setenumerate{label=\textnormal{(\arabic*)}}

\usepackage{amsmath,amssymb,dsfont,verbatim,bm,mathtools,geometry,fge}
\usepackage[mathcal]{euscript}
\usepackage[raggedright]{titlesec}
\usepackage{flafter}
\usepackage{microtype}
\usepackage{quoting}
\usepackage[latin1]{inputenc}
\usepackage{setspace}
\usepackage{mathtools}
\usepackage{tikz}
\usepackage{float,subfig}
\usepackage{amsrefs}
\usepackage[linesnumbered,ruled,vlined]{algorithm2e}
\SetAlgoSkip{smallskip}
\SetAlgoInsideSkip{smallskip}

\usepackage{booktabs}
\usepackage[section]{placeins}
\usepackage{flafter}
\usepackage{longtable, tabu}
\usepackage[foot]{amsaddr}

\newcommand{\ra}[1]{\renewcommand{\arraystretch}{#1}}

\titleformat{\chapter}[display]
{\normalfont\huge\bfseries}{\chaptertitlename\\thechapter}{20pt}{\Huge}
\titleformat{\section}
{\normalfont\Large\bfseries\center}{\thesection}{1em}{}
\titleformat{\subsection}
{\normalfont\large\bfseries}{\thesubsection}{1em}{}
\titleformat{\subsubsection}[runin]
{\normalfont\normalsize\bfseries}{\thesubsubsection}{1em}{}
\titleformat{\paragraph}[runin]
{\normalfont\normalsize\bfseries}{\theparagraph}{1em}{}
\titleformat{\subparagraph}[runin]
{\normalfont\normalsize\bfseries}{\thesubparagraph}{1em}{}
\titlespacing*{\chapter} {0pt}{50pt}{40pt}
\titlespacing*{\section} {0pt}{3.5ex plus 1ex minus .2ex}{2.3ex plus .2ex}
\titlespacing*{\subsection} {0pt}{3.25ex plus 1ex minus .2ex}{1.5ex plus .2ex}
\titlespacing*{\subsubsection}{0pt}{3.25ex plus 1ex minus .2ex}{1.5ex plus .2ex}
\titlespacing*{\paragraph} {0pt}{3.25ex plus 1ex minus .2ex}{1em}
\titlespacing*{\subparagraph} {\parindent}{3.25ex plus 1ex minus .2ex}{1em}
\usepackage{titletoc}
\usepackage{mathrsfs}
\contentsmargin{2.5em}
\dottedcontents{section}[1.5em]{}{2.8em}{1pc}
\dottedcontents{subsection}[3.7em]{}{3.4em}{1pc}
\dottedcontents{subsubsection}[7.6em]{}{3.2em}{1pc}
\dottedcontents{paragraph}[10.3em]{}{3.2em}{1pc}

\newtheorem{theorem}{Theorem}[section]
\newtheorem{lemma}[theorem]{Lemma}
\newtheorem{proposition}[theorem]{Proposition}

\theoremstyle{definition}
\newtheorem{definition}[theorem]{Definition}

\newtheorem{notation}[theorem]{Notation}

\theoremstyle{remark}
\newtheorem{remark}[theorem]{Remark}

\DeclareMathOperator{\Aut}{Aut}

\DeclareMathOperator{\gap}{gap}

\DeclareMathOperator{\Supp}{Supp}

\DeclareMathOperator{\en}{en}

\DeclareMathOperator{\BezoutCoefficients}{BezoutCoefficients}
\DeclareMathOperator{\NumberOfFactors}{NumberOfFactors}

\DeclareMathOperator{\Ss}{S}
\DeclareMathOperator{\st}{st}

\DeclareMathOperator{\Pred}{Pred}

\DeclareMathOperator{\lcm}{lcm}
\DeclareMathOperator{\Dir}{Dir}
\DeclareMathOperator{\dir}{dir}

\DeclareMathOperator{\Simple}{Simple}
\DeclareMathOperator{\length}{length}
\newcommand{\ov}{\overline}

\DeclareMathOperator{\lo}{lo}
\DeclareMathOperator{\hi}{hi}
\DeclareMathOperator{\GetChildrenAndFinalList}{GetChildrenAndFinalList}

\DeclareMathOperator{\GetStartingEdges}{GetStartingEdges}

\DeclareMathOperator{\mnFamilies}{mnFamilies}

\DeclareMathOperator{\GetGeneratedCorners}{GetGeneratedCorners}
\DeclareMathOperator{\GeneratedCorners}{GeneratedCorners}
\DeclareMathOperator{\GetCompleteChains}{GetCompleteChains}
\DeclareMathOperator{\AdmissibleCompleteChains}{AdmissibleCompleteChains}
\DeclareMathOperator{\StartingEdges}{StartingEdges}
\DeclareMathOperator{\FinalList}{FinalList}
\DeclareMathOperator{\GetCornerChildrenList}{GetCornerChildrenList}
\DeclareMathOperator{\CornerChildrenList}{CornerChildrenList}
\DeclareMathOperator{\ChildrenList}{ChildrenList}
\DeclareMathOperator{\GetPossibleLastLowerCorners}{GetPossibleLastLowerCorners}

\DeclareMathOperator{\GetIsAdmissible}{GetIsAdmissible}
\DeclareMathOperator{\PFL}{PFL}
\DeclareMathOperator{\PLLC}{PLLC}

\DeclareMathOperator{\Lmax}{Lmax}
\DeclareMathOperator{\gmax}{gmax}
\DeclareMathOperator{\enF}{enF}
\DeclareMathOperator{\CompleteChains}{CompleteChains}
\DeclareMathOperator{\OpenChains}{OpenChains}
\DeclareMathOperator{\POpenChains}{POpenChains}

\DeclareMathOperator{\MN}{MN}
\DeclareMathOperator{\TRUE}{TRUE}
\DeclareMathOperator{\FALSE}{FALSE}
\DeclareMathOperator{\Last}{Last}

\newcommand{\wrs}{\hspace{-0.9pt}\wr\hspace{-0.9pt}}

\DeclareMathOperator{\Primefactors}{Prime factors}
\DeclareMathOperator{\IsAdmissible}{IsAdmissible}

\begin{document}

\title{Some algorithms related to the Jacobian Conjecture}

\author[Jorge A. Guccione]{Jorge A. Guccione$^{1,}$$^2$}
\address{$^1$ Universidad de Buenos Aires. Facultad de Ciencias Exactas y Naturales. Departamento de Matem\'atica. Buenos Aires. Argentina}
\address{$^2$ CONICET-Universidad de Buenos Aires. Instituto de Investigaciones Matem\'aticas ``Luis A. santal\'o'' (IMAS). Buenos Aires. Argentina}
\email{vander@dm.uba.ar}

\author[Juan J. Guccione]{Juan J. Guccione$^{1,}$$^3$}
%\address{$^1$ Universidad de Buenos Aires. Facultad de Ciencias Exactas y Naturales. Departamento de Matem\'atica. Buenos Aires. Argentina}
\address{$^3$ CONICET. Instituto Argentino de Matem\'atica (IAM). Buenos Aires. Argentina}
\email{jjgucci@dm.uba.ar}

\thanks{Jorge A. Guccione and Juan J. Guccione were supported by UBACyT 20020150100153BA (UBA) and PIP 11220110100800CO (CONICET)}

\author[Rodrigo Horruitiner]{Rodrigo Horruitiner$^{4}$}
%\address{Pontificia Universidad Cat\'olica del Per\'u, Secci\'on Matem\'aticas, PUCP, Av. Universitaria 1801, San Miguel, Lima 32, Per\'u.}
\email{rhorruitiner@pucp.edu.pe}

\author[Christian Valqui]{Christian Valqui$^{4,}$$^5$}
\address{$^4$Pontificia Universidad Cat\'olica del Per\'u, Secci\'on Matem\'aticas, PUCP,
Av. Universitaria 1801, San Miguel, Lima 32, Per\'u.}

\address{$^5$Instituto de Matem\'atica y Ciencias Afines (IMCA) Calle Los Bi\'ologos 245. Urb San C\'esar.
La Molina, Lima 12, Per\'u.}
\email{cvalqui@pucp.edu.pe}
\thanks{Christian Valqui was supported by PUCP-DGI-CAP-2016-329.}

\begin{abstract}
We describe an algorithm that computes possible corners of hypothetical counterexamples to the Jacobian Conjecture up to a given bound. Using this algorithm we compute the possible families corresponding to $\gcd(\deg(P),\deg(Q))\le 35$, and all the pairs $(\deg(P),\deg(Q))$ with $\max(\deg(P),\deg(Q))\le 150$ for any hypothetical counterexample.
\end{abstract}

\maketitle

\setcounter{tocdepth}{2}
\tableofcontents
%\listofalgorithms

\begin{spacing}{0.95}

\section*{Introduction}

Let $K$ be a characteristic zero field and let $L\coloneqq K[x,y]$ be the polynomial algebra in two indeterminates. The Jacobian Conjecture (JC) in dimension two stated by Keller in \cite{K} says that any pair of polynomials $P,Q\in L$ with $[P,Q]\coloneqq \partial_x P \partial_y Q - \partial_x Q \partial_y P\in K^{\times}$ defines an automorphism $f$ of $L$ via $f(x)\coloneqq P$ and $f(y)\coloneqq Q$. If this conjecture is false, then there exist $P,Q\in L$ such that $[P,Q]=K^{\times}$, and there exist $m,n,a,b\in \mathds{N}$, such that $m,n>1$ are coprime, $a<b$, the support of $P$ is contained in the rectangle with vertices $\{(0,0), m(a,0),m(a,b),m(0,b)\}$, the support of $Q$ is contained in the rectangle with vertices $\{(0,0),n(a,0),n(a,b),n(0,b)\}$, the point $m(a,b)$ is in the support of $P$ and the point $n(a,b)$ is in the support of $Q$. Note that $\deg(P)=m(a+b)$ and $\deg(Q)=n(a+b)$.

In~\cite{H} Heitmann establishes several restrictions on these possible corners $(a,b)$ and in \cite{H}*{Theorem~2.24} he determines various of these possible corners $(a,b)$. Moreover in~\cite{H}*{Theorem~2.25}, for some of these corners, he finds families $\{(r+sj,t+uj):j\in\mathds{N}\}$ of admissible pairs $(m,n)$. These corners were also found in~\cite{GGV1}*{Remark~7.14}, using more elementary methods and discrete geometry on the plane. In both articles the lists of possible corners where given without a formal proof, referring to a computer program.

In \cite{GGV2} we found more conditions on the points $(a,b)$, and in this article we present an algorithm that generates the list of points satisfying all the conditions up to a fixed upper bound for $a+b$. Naturally this list is included in the one found in~\cite{GGV1}*{Remark~7.14}. The algorithm also determines the families of admissible pairs $(m,n)$, for each of these corners.

In order to exploit the simple geometric ideas of our method we also present a graphic interface of the program which includes all the filters and allows the user to grasp in detail if and why a certain corner is admissible or not.

At the end we list all possible corners $(a,b)$ with $a\!+\!b\!<\! 36$, and their corresponding $(m,n)$-families. Furthermore if $(P,Q)$ is a counterexample to the Jacobian Conjecture that satisfy the inequality $\gcd(\deg(P),\deg(Q))< 36$, then we give additional information on the Newton polygons of $P$ and~$Q$. We also provide the same information for the counterexamples that satisfy $\max\{\deg(P),\deg(Q)\}\le~150$.

\smallskip

Along this paper we will freely use the notations of \cite{GGV1}.

\section{Restrictions on possible last lower corners}

The first step in our strategy is to construct a set of points in $\mathds{N}_0\times \mathds{N}_0$, that includes all the possible last lower corners (see \cite{GGV2}*{Definition~3.17}).

\smallskip

\begin{definition}
Let $(a,b)\in \mathds{N}\times \mathds{N}_0$ and $(\rho,\sigma)\in\mathfrak{V}\cap \hspace{0.7pt}[(0,-1),(1,-1)[\hspace{0.7pt}$ (see \cite{GGV1}*{Definition~1.5}). We say that $((a,b),(\rho,\sigma))$ is a {\em possible final pair} if one of the following conditions is fulfilled:

\begin{enumerate}[itemsep=0.5ex,topsep=0.5ex]

\item $b=0$ and $(\rho,\sigma) = (0,-1)$,

\item there exists an admissible chain of length $k\in\mathds{N}$ (see~\cite{GGV2}*{Definition~3.15})
$$
\qquad\quad\mathfrak{C}=\bigl((C_j)_{j\in\{0,\dots,k\}},(R_j)_{j\in\{1,\dots,k\}},(\rho_j,\sigma_j)_{j\in\{1,\dots,k\}}\bigr),
$$
with $C_k = (a,b)$ and $(\rho_k,\sigma_k) = (\rho,\sigma)$.

\end{enumerate}
\end{definition}

\begin{remark}
Recall from~\cite{GGV2}*{Definition~3.17} that if $((a,b),(\rho,\sigma))$ is a possible final pair, then $(a,b)$ is said to be a possible last lower corner.
\end{remark}

\begin{remark}
By~\cite{GGV2}*{Definition~3.15(6)}, if $((a,b),(\rho,\sigma))$ is a possible final pair, then $b<a$.
\end{remark}

\begin{remark}
By~\cite{GGV2}*{Remark~3.19}, we know that if $a > 2b> 0$, then $((a,b),(1,-2))$ is a possible final pair.
\end{remark}

\begin{remark}
By~\cite{GGV2}*{Proposition~3.25}, if $(a,b)$ is a possible last lower corner, then $b\le (a-b-1)^2$, which, since $a\ge 1$ and $b<a$, is equivalent to $b\le \frac{1}{2} \left(2 a-\sqrt{4 a-3}-1\right)$. %This is used in line~3 of Algorithm~\ref{PLLC}.
\end{remark}

%\begin{remark}\label{prop3.29}
%By~\cite{GGV2}*{Proposition~3.29}, if $(a,b)$ is a possible last lower corner, then $b + \gcd(a,b) \ne a$.
%\end{remark}

\begin{proposition}\label{cotasjacobiano2.12}
If $((a,b),(\rho,\sigma))$ is a possible final pair with $b\!>\!0$ and $a\!\le\!2b$, then \hbox{$v_{\rho,\sigma}(a,b)\!\ge\!\rho$} and there exist a possible final pair $((r,s),(\rho',\sigma'))$ such that:

\begin{enumerate}[itemsep=0.5ex,topsep=0.5ex]

\item $r<a$, $s<b$ and $r-s<a-b$,

\item $v_{\rho,\sigma}(r,s) = v_{\rho,\sigma}(a,b)$,

%\item $-\dfrac{\rho + \sigma}{v_{\rho,\sigma}(a,b)}\ov{\vartheta} \in \mathds{N}$,

\item $\ov{\vartheta} \leq \gcd(a - r, b - s)$ or $\ov{\vartheta}\mid \gcd(r,s)$, where $\ov{\vartheta}\coloneqq \frac{\rho a+\sigma b}{\gcd(\rho+\sigma,\rho a+\sigma b)}$.

\end{enumerate}
\end{proposition}

\begin{proof}
By hypothesis there exists an admissible chain
$$
\mathfrak{C}=\bigl((C_j)_{j\in\{0,\dots,k\}},(R_j)_{j\in\{1,\dots,k\}},(\rho_j,\sigma_j)_{j\in\{1,\dots,k\}}\bigr) \qquad\text{with $C_k = (a,b)$ and $(\rho_k,\sigma_k) = (\rho,\sigma)$.}
$$
Note that $k\ge 1$ and set
$$
(r,s) \coloneqq C_{k-1}\qquad\text{and}\qquad (\rho',\sigma')\coloneqq \begin{cases}(\rho_{k-1},\sigma_{k-1}) & \text{if $k>1$,}\\(0,-1) & \text{if $k=1$.} \end{cases}
$$
By \cite{GGV2}*{Definition~3.15(7)} we know that $v_{\rho,\sigma}(a,b) \ge \rho$. We next prove the rest of the proposition. Item~(1) follows from~\cite{GGV2}*{Remark~3.16}, while item~(2) follows from items~(4) and (5) of~\cite{GGV2}*{Def\-inition~3.15}. Moreover, by items~(7) and~(8) of~\cite{GGV2}*{Def\-inition~3.15}, the hypothesis of~\cite{GGV2}*{Propo\-sition~3.12} are satisfied with $R=R_k$. Since $a \leq 2b$, case~(1) of that proposition is impossible. Let $\theta$ and $t'$ be as in~\cite{GGV2}*{Proposition~3.12}. By~\cite{GGV2}*{Remark~3.13}
$$
\frac{\vartheta}{t'}=-\frac{v_{\rho,\sigma}(R)}{\rho+\sigma}  = -\frac{\rho a+\sigma b}{\rho+\sigma}.
$$
Hence $\ov{\vartheta}\mid \vartheta$, and so item~(3) follows from items~(2) and~(3) of~\cite{GGV2}*{Proposition~3.12}.
\end{proof}

Based on the previous results in Algorithm~\ref{PLLC} we present a method for the generation of a set $\PLLC$ that includes all possible last lower corners $(a,b)$ with $a\le x_{max}$ for a given $x_{max}$. In the algorithm we use an auxiliary list $\PFL$.

\begin{center}
\begin{algorithm}[H]
%  \SetAlgoLined
\DontPrintSemicolon
\KwIn{Maximum x coordinate value $x_{max} > 0$.}
\KwOut{A list $\PLLC$, that includes all the possible last lower corners $(a,b)$ with $a\le x_{max}$.}
%$\PLLC \gets \emptyset$\;
%$\PFL \gets \emptyset$\;
\For{$a\gets 1$ \KwTo $x_{max}$}{
    $b\gets 0$\;
    \While{$b\le \frac{1}{2} \left(2 a-\sqrt{4 a-3}-1\right)$}{
        \uIf{$b=0$}{$(\rho,\sigma)_{a,b}\gets (0,-1)$, add $((a,b),(\rho,\sigma)_{a,b})$ to $\PFL$ and add $(a,b)$ to $\PLLC$\;}
        \uElseIf{$a > 2b>0$}{
        $(\rho,\sigma)_{a,b}\gets (1,-2)$, add $((a,b),(\rho,\sigma)_{a,b})$ to $\PFL$ and add $(a,b)$ to $\PLLC$\;
        }
        \Else{set $(\rho,\sigma)_{a,b}\coloneqq (1,-1)$ \;
            \For{$\bigl((r,s),(\rho,\sigma)_{r,s}\bigr)$ \normalfont{\textrm{in}} $\PFL$ \normalfont{\textrm{such that}} $r<a$, $s<b$ \normalfont{\textrm{and}} $r-s<a-b$ }{
                $N_1 \gets \gcd(a-r, b-s)$\;
                $N_2 \gets \gcd(r,s)$ \;
                $(\rho, \sigma) \gets \dfrac{1}{N_1}(b-s, r-a)$ \;
                $g \gets \gcd(\rho + \sigma, \rho a + \sigma b)$ \;
                $\ov{\vartheta} \gets \dfrac{\rho a + \sigma b}{g}$ \;
                \If{
                    $(\rho,\sigma)_{r,s}<(\rho,\sigma)<(\rho,\sigma)_{a,b}$, $v_{\rho,\sigma}(a,b) \ge  \rho$ {\normalfont \textrm{and}}
                    $(\ov{\vartheta} \leq N_1$ {\normalfont \textrm{or}} $\ov{\vartheta} \mid N_2$)
                }{
                    $(\rho,\sigma)_{a,b}\gets (\rho,\sigma)$\;
                }
            }
                \If{$(\rho,\sigma)_{a,b}<(1,-1)$
                }{
                    add $\bigl((a,b),(\rho,\sigma)_{a,b}\bigr)$ to $\PFL$ and add $(a,b)$ to $\PLLC$\;
                }
        }
   $b\gets b+1$\; }
}
\Return{$\PLLC$.}
\caption{GetPossibleLastLowerCorners}
\label{PLLC}
\end{algorithm}
\end{center}

\end{spacing}

\begin{spacing}{0.975}

\section{Construction of admissible complete chains up to a certain bound}

Assume that the Jacobian Conjecture is false and define
\begin{equation}\label{def B}
B\coloneqq \min\bigl\{\gcd(v_{1,1}(P),v_{1,1}(Q)):\text{where $(P,Q)$ runs on the counterexamples of J.C.}\bigr\}.
\end{equation}
Then, by \cite{GGV1}*{Corollary~5.21} there exists a counterexample $(P,Q)$ and $m,n\in \mathds{N}$ coprime such that $(P,Q)$ is a standard $(m,n)$-pair and a minimal pair (that is, the greatest common divisor of $v_{11}(P)$ and $v_{11}(Q)$ is $B$). Let $A_0$ be as in Remark~\ref{rem 2.21}. By \cite{GGV1}*{Proposition~5.2 and Corollary~5.21(3)}
$$
A_0 = \frac{1}{m}\en_{10}(P)  \quad\text{and}\quad \gcd(v_{11}(P),v_{11}(Q)) = v_{11}(A_0).
$$
This point $A_0$ corresponds to $(a,b)$ in the introduction. In Theorem~\ref{standard pair generates complete chain} below, we obtain a chain
\begin{equation*}
(\mathcal{C}_0,\dots,\mathcal{C}_j,\mathcal{A}_{j+1}) = \bigl((\mathcal{A}_0,\mathcal{A}_0'),\dots, (\mathcal{A}_j,\mathcal{A}_j'),\mathcal{A}_{j+1}\bigr),
\end{equation*}
such that $A_0$ is the geometric realization of $\mathcal{A}_0$ (see Definition~\ref{realizacion geometrica}), and that satisfies (among others) certain geometric conditions, which are codified in Definition~\ref{complete chain}. Then, we show that this chain also satisfies certain arithmetic conditions (see the comment below Definition~\ref{cond div}). The chains meeting the requirements of Definitions~\ref{complete chain} and~\ref{cond div} are called admissible complete chains. In Algorithm~\ref{Main algorithm} we construct all the admissible complete chains that satisfy $v_{11}(A_0)\le M$ for a given positive integer bound $M$.

By Theorem~\ref{standard pair generates complete chain}
and Remark~\ref{admis} we know that $\mathcal{A}_0$ is the first coordinate of $\mathcal{C}_0$ for one of the admissible complete chains $(\mathcal{C}_0,\dots,\mathcal{C}_j,\mathcal{A}_{j+1})$ obtained running Algorithm~\ref{Main algorithm} with $M \ge B$. For example we obtain immediately that the Jacobian Conjecture is false, then $B\ge 16$, since there are no admissible complete chains with $v_{11}(A_0)< 16$ (this result was already obtained in \cite{GGV1}). More importantly, we will see that many of the admissible complete chains obtained in Algorithm~\ref{GetCompleteChains} can not come from a standard $(m,n)$-pair as in Theorem~\ref{standard pair generates complete chain}.

\subsection{Valid edges}

In this subsection and in the next one we introduce the basic ingredients for the definition and construction of the complete chains.

\smallskip

For each $l\in \mathds{N}$ we let $\mathds{N}_{\!(l)}$ denote the set $\{(a,l): a\in \mathds{N}\}$. In the sequel we will write $a\wrs l$ instead of $(a,l)$. Moreover we will use the notation $I\coloneqq ](1,-1),(1,0)]$.

\begin{definition}\label{realizacion geometrica}
A {\em corner} is a pair $(a\wrs l,b)$ with $a\wrs l\in \mathds{N}_{\!(l)}$ and $b\in \mathds{N}_0$. For $l=1$ we will write $(a,b)$ instead of $(a\wrs 1,b)$. The {\em geometric realization} of a corner $\mathcal{A}=(a\wrs l,b)$ is the point $A\coloneqq \bigl(\frac{a}{l},b\bigr) \in \frac{1}{l}\mathds{N}\times \mathds{N}_0$.
\end{definition}

\smallskip

Let $l\in \mathds{N}$. In the rest of this section given $\mathcal{A},\mathcal{A}'\in \mathds{N}_{\!(l)}\times\mathds{N}_0$ with $\mathcal{A}\ne \mathcal{A}'$, we write
$$
\mathcal{A} = (a\wrs l,b),\quad \mathcal{A}' = (a'\wrs l,b'), \quad (\rho,\sigma)\coloneqq \dir (A-A') \quad\text{and}\quad \gap(\rho,l)\coloneqq \frac{\rho}{\gcd(\rho,l)}.
$$

\begin{definition}\label{valid edges}
Set $d\coloneqq \gcd(a,b)$, $\ov a\coloneqq \frac{a}{d}$ and $\ov b\coloneqq \frac{b}{d}$, The pair $(\mathcal{A},\mathcal{A}')$ is called a {\em valid edge} if

\begin{enumerate}[itemsep=1.0ex,topsep=1.0ex]

\item $(\rho,\sigma)\in I$,

\item $v_{1,-1}(A')\ne 0$, $v_{1,-1}(A)<0$ and $v_{1,-1}(A)<v_{1,-1}(A')$,

\item there exist $\mathcal{\enF}\in\mathds{N}_{\!(l)}\times\mathds{N}$ and $\mu\in \mathds{N}$, with $\mu\le l(bl-a)+1/\ov b$ and $d\nmid \mu$, such that
$$
\mathcal{\enF}=\frac{\mu}{d} \mathcal{A}\coloneqq \mu(\ov a\wrs l,\ov b),\quad v_{\rho,\sigma}(\enF) = \rho+\sigma \quad\text{and}\quad \text{if $l=1$, then $\mu<d$.}
$$

\item If $l=1$ and $v_{1,-1}(A')>0$, then $A'$ is a possible last lower corner.

\end{enumerate}
The valid edge $(\mathcal{A},\mathcal{A}')$ is called {\em simple} if $v_{01}(\enF)-1=\gap(\rho,l)$ and ($\gap(\rho,l)>1$ or $v_{01}(A')>0$).
\end{definition}

\begin{remark}\label{zazaza}
By item~(1) the last inequality in item~(2) is equivalent to $v_{01}(A-A')>0$. Moreover $d>1$ since $d\nmid \mu$. We can also replace condition~(3) by

\begin{enumerate}[itemsep=1.0ex,topsep=1.0ex]

\item[(3')] $\exists\ \mu\in \mathds{N}$, such that $\frac{\mu}{d}=\frac{\rho+\sigma}{v_{\rho,\sigma}(A)}$, $\mu\le l(bl-a)+1/\ov b$, $d\nmid \mu$ and if $l=1$, then $\mu<d$.

\end{enumerate}
Moreover, such a $\mu$ univocally determines $\mathcal{\enF}$ via the equality $\mathcal{\enF} = \frac{\mu}{d}\mathcal{A}$. Write $\mathcal{\enF}=(f_1\wrs l,f_2)$. Since $v_{\rho,\sigma}(\enF) = \rho+\sigma$ and $f_2\ge 1$,
$$
(\rho,\sigma) = \frac{1}{\gcd(f_1-l,f_2l-l)}(f_2l-l,l-f_1).
$$
This equality implies $f_2>1$, because by condition~(1) we have $\rho>0$. Thus, by~\cite{GGV2}*{Remark~3.9} we know that
$$
\gap(\rho,l) = \frac{f_2-1}{\gcd(f_1-l,f_2-1)}.
$$
Consequently $v_{01}(\enF) -1= \gap(\rho,l)$ if and only if $\gcd(f_1-l,f_2-1)=1$.
\end{remark}

\begin{notation}
Fixed $l\in \mathds{N}$ and given $A = \bigl(\frac{a}{l},b\bigr)\in \frac{1}{l}\mathds{N}\times \mathds{N}_0$ we set $\mathcal{A}\coloneqq (a\wrs l,b)\in \mathds{N}_{\!(l)}\times \mathds{N}_0$.
\end{notation}

In Algorithm~\ref{GetStartingEdges} we obtain a list $\StartingEdges$ consisting of all valid edges $(\mathcal{A},\mathcal{A}')$ starting with a given $A\in \mathds{N}\times \mathds{N}$ such that $v_{1,-1}(A)<0$. We use freely the results of Remark~\ref{zazaza}. Before running this algorithm with input a corner $\mathcal{A}=(a,b)$ it is necessary to run Algorithm~\ref{PLLC} with input greater than or equal to $a$, in order to obtain a list $\PLLC$.

\begin{center}
\begin{algorithm}[H]
  %\SetAlgoLined
  \DontPrintSemicolon
  %\SetKwFunction{MainInductiveStep}{MainInductiveStep}
  %  \SetKwProg{myalg}{Algorithm}{}{}
  %  \SetKwProg{myproc}{Procedure}{}{}
  %  \myproc{\MainInductiveStep{$\mathcal{A},\mathcal{A}'$}}{
     \KwIn{A corner $A=(a,b)\in \mathds{N}\times\mathds{N}$ with $a<b$, and a list $\PLLC$.}
     \KwOut{A list $\StartingEdges$, consisting of all valid edges $(\mathcal{A},\mathcal{A}')$.}	
    $d \gets \gcd(a,b)$\;
    \For{$\mu=1$ \KwTo $d-1$}{
        $\enF \gets \frac{\mu}{d}(a,b)$\;
        $(\rho, \sigma) \gets \dir(\enF-(1,1))$\;
            \For{$i=1$ \KwTo $\Bigl\lfloor \frac{b}{\rho}\Bigr\rfloor$}{
            $A'\gets (a,b)-i(-\sigma,\rho)$\;
            \If{$v_{1,-1}(A')< 0$ \normalfont\textbf{ or }
              ( $v_{1,-1}(A')> 0$\normalfont\textbf{ and } $A'\in \PLLC$)}{
                 \bf{add} $(\mathcal{A},\mathcal{A}')$ \bf{to} $\StartingEdges$\;
                }
            }
        }
       % }
 {\bf RETURN} $\StartingEdges$
 \caption{GetStartingEdges}
\label{GetStartingEdges}
\end{algorithm}
\end{center}

In the following proposition we show among other things how a regular corner of an $(m,n)$-pair $(P,Q)$ gives rise to a valid edge.

\begin{proposition}\label{multiplicidad}
Let $l\ge 1$ and let $(P,Q)$ be an $(m,n)$-pair in $L^{(l)}$. Assume that if $l=1$, then $(P,Q)$ is a standard $(m,n)$-pair in $L$ (see \cite{GGV1}*{Definition~4.3}). Let $(A,(\rho,\sigma))$ be a regular corner  of $(P,Q)$ (see~\cite{GGV1}*{Definition~5.5}) and let $A'\coloneqq \frac{1}{m}\st_{\rho,\sigma}(P)$. Write
$$
\ell_{\rho,\sigma}(P) = x^{m\frac{a'}{l}} y^{mb'} p(z)\quad\text{with $z\coloneqq x^{-\frac{\sigma}{\rho}}y$, $p\in K[z]$ and $p(0)\ne 0$.}
$$
The following facts hold:

\begin{enumerate}[itemsep=1.0ex,topsep=1.0ex]

\item If $l = 1$, then the regular corner $(A,(\rho,\sigma))$ is of type~II.

\item If $(A,(\rho,\sigma))$ is of type~II (see the comments above~\cite{GGV1}*{Definition~5.9}), then $(\mathcal{A},\mathcal{A}')$ is a valid edge.

\item If $\lambda\in K^{\times}$ is a root of $p$, then
$$
\frac{m_{\lambda}}{m}\le \frac{v_{01}(A-A')}{\gap(\rho,l)},\quad\text{where $m_{\lambda}$ denotes the multiplicity of $\lambda$.}
$$
If moreover $(\mathcal{A},\mathcal{A}')$ is simple, then $\frac{m_{\lambda}}{m}= \frac{v_{01}(A-A')}{\gap(\rho,l)}$.

\item If $(A,(\rho,\sigma))$ is of type~II.b), then there exists a root $\lambda\in K^{\times}$ of $p$ such that
    \begin{equation}\label{5.16}
    \quad\qquad b'< \frac{\rho a + \sigma b l}{l(\rho+\sigma)} \le \frac{m_{\lambda}}{m},
    \end{equation}
where $m_{\lambda}$ denotes the multiplicity of $\lambda$ in $p$.

\end{enumerate}
\end{proposition}

\begin{proof} 1)\enspace By~\cite{GGV1}*{Remark~5.10 and Propositions~5.22 and~6.1}.

\smallskip

\noindent 2)\enspace First note that by~\cite{GGV1}*{Remark~1.8} we have $A\in \frac{1}{l}\mathds{N}\times \mathds{N}_{(0)}$. We now check that the pair $(\mathcal{A},\mathcal{A}')$ satisfies conditions~(1)--(4) of Definition~\ref{valid edges}. The fact that $(\rho,\sigma)\in~I$ and the inequality $v_{1,-1}(A)<0$ follow from \cite{GGV1}*{Definition~5.5}). Moreover, $v_{1,-1}(A')\ne 0$ by \cite{GGV1}*{Corollary~5.7(1) and Theorem~2.6(4)}, while $v_{1,-1}(A)<v_{1,-1}(A')$ by Remark~\ref{zazaza}, because $v_{01}(A')<v_{01}(A)$. So conditions~(1) and~(2) are true. Let $\mu$ and $F$ be as in~\cite{GGV1}*{Proposition~5.14} and set $\enF\coloneqq \en_{\rho,\sigma}(F)$. All the assertions in condition~(3), with the exception of the last one, follow from the definition of $\mu$ and items~(3) and~(4) of that proposition. Assume now $l=1$ (which by hypothesis implies that $P,Q\in L$). By \cite{vdE}*{Theorem~10.2.1 and Proposition~10.2.6} there exists $k\in \mathds{N}$ such that $(km,0)\in \Supp(P)$. So
$$
v_{\rho,\sigma}(A)=\frac 1m v_{\rho,\sigma}(P)\ge \frac 1m v_{\rho,\sigma}(km,0)=k\rho\ge\rho\ge \rho+\sigma = v_{\rho,\sigma}(\enF).
$$
Since $\mu v_{\rho,\sigma}(A) = d v_{\rho,\sigma}(\enF)$ and $d\nmid \mu$, this implies that $\mu<d$. We finally prove item~(4). Since $(A,(\rho,\sigma))$ is of type~II and $v_{1,-1}(A')>0$, it is of type~II.b). Consequently if $l=1$ it follows from \cite{GGV1}*{Remark~6.3} that $(A,A',(\rho,\sigma))$ is the starting triple of $(P,Q)$ (see~\cite{GGV1}*{Definition~6.2}), and so condition~(4) is true by~\cite{GGV2}*{Remark~3.23}, because by hypothesis $P,Q\in L$.

\smallskip

\noindent 3)\enspace Let $F$ be as in \cite{GGV1}*{Theorem~2.6} and write
$$
F = x^{\frac{u}{l}} y^v f(z)\quad\text{with $z\coloneqq x^{-\frac{\sigma}{\rho}}y$, $f\in K[z]$ and $f(0)\ne 0$.}
$$
By \cite{GGV2}*{Remark~3.9} there exist $\ov{p},\ov{f}\in K[z]$ such that
$$
p(z) = \ov{p}(z^k)\quad\text{and}\quad f(z) = \ov{f}(z^k),\qquad\text{where $k\coloneqq \gap(\rho,l)$.}
$$
So,
$$
t\coloneqq \deg \ov p= \frac{\deg p}{k} =\frac{v_{01}(\en_{\rho,\sigma}(P)-\st_{\rho,\sigma}(P))}{k}=m \frac{v_{01}(A-A')}{k}.
$$
By~\cite{GGV2}*{Remark~3.8} we have $m_{\lambda}\le \deg \ov{p}$, which yields $\frac{m_{\lambda}}{m}\le \frac{v_{01}(A-A')}{\gap(\rho,l)}$. Assume now that $(\mathcal{A},\mathcal{A}')$ is simple. Since $k = v_{01}(\en_{\rho,\sigma}(F))-1$, we have
$$
k+1 = v_{01}(\en_{\rho,\sigma}(F)) = v_{01}(F) = v + \deg(f) = v + k\deg(\ov{f}),
$$
which implies $\deg(\ov{f}) = v = 1$ or $k=1$, $v=0$ and $\deg(\ov{f}) = 2$. But if $v=0$, then by~\cite{GGV1}*{The\-orem~2.6(2)}
$$
\Bigl(\frac{u}{l},0\Bigr) = \st_{\rho,\sigma} (F) \sim A',
$$
which is impossible since $v_{01}(A')> 0$, since $k=0$ and $(\mathcal{A},\mathcal{A}')$ is simple. Hence, $\deg(\ov f)=1$ and so, by~\cite{GGV1}*{Proposition~2.11(3)} we have $\ov p(z^k)=(z^k-c)^t$ for some constant $c\in K^{\times}$. Consequently, by~\cite{GGV2}*{Remark~3.8}, every linear factor of $p$ has multiplicity $t$. Thus $m_{\lambda}=t=m\frac{v_{01}(A-A')}{\gap(\rho,l)}$, as desired.

\smallskip

\noindent 4)\enspace By \cite{GGV1}*{Proposition~5.16} there exists $\lambda\in K^{\times}$ such that the second inequality in~\eqref{5.16} is true. Since $\rho>0$ and $\frac{a'}{l}-b'>0$, we have
$$
\Bigl(\rho\frac{a'}{l} + \sigma b'\Bigr) - (\rho+\sigma)b' =  \rho \Bigl(\frac{a'}{l} - b'\Bigr) > 0.
$$
Since $\rho+\sigma>0$ and $v_{\rho,\sigma}(A) = v_{\rho,\sigma}(A')$, this implies the first inequality in~\eqref{5.16}.
\end{proof}

\begin{remark}\label{comentarios0} Let $l\ge 1$ and let $(P,Q)$ be an $(m,n)$-pair in $L^{(l)}$. Let $(A,(\rho,\sigma))$ be a regular corner  of $(P,Q)$ and let $A'\coloneqq \frac{1}{m}\st_{\rho,\sigma}(P)$. Write
$$
\ell_{\rho,\sigma}(P) = x^{m\frac{a'}{l}} y^{mb'} p(z)\quad\text{with $z\coloneqq x^{-\frac{\sigma}{\rho}}y$, $p\in K[z]$ and $p(0)\ne 0$,}
$$
If $(A,(\rho,\sigma))$ is of type~I, then all the roots of $p$ are simple. In fact if $p(z) = (z-\lambda)^2\tilde{p}(z)$, then
\begin{align*}
[\ell_{\rho,\sigma}(P),\ell_{\rho,\sigma}(Q)] & = [x^{m\frac{a'}{l}}y^{mb'}(z-\lambda)^2\tilde{p}(z) ,\ell_{\rho,\sigma}(Q)]\\
& = 2(z-\lambda)x^{m\frac{a'}{l}}y^{mb'}\tilde{p}(z) [(z-\lambda),\ell_{\rho,\sigma}(Q)] + (z-\lambda)^2[x^{m\frac{a'}{l}}y^{mb'}\tilde{p}(z) ,\ell_{\rho,\sigma}(Q)],
\end{align*}
which contradicts the fact that $[\ell_{\rho,\sigma}(P),\ell_{\rho,\sigma}(Q)]\in K^{\times}$.
\end{remark}

\begin{remark}\label{comentarios}
Let $l\ge 1$ and let $(P,Q)$ be an $(m,n)$-pair in $L^{(l)}$. Let $(A,(\rho,\sigma))$ be a regular corner  of $(P,Q)$ and let $A'\coloneqq \frac{1}{m}\st_{\rho,\sigma}(P)$. Write
$$
\ell_{\rho,\sigma}(P) = x^{\frac{k}{l}} \mathfrak{p}(z)\quad\text{where $z\coloneqq x^{-\frac{\sigma}{\rho}}y$ and $\mathfrak{p}(z)\in K[z]$.}
$$
Let $\lambda\in K^{\times}$ be a root of $\mathfrak{p}$ of multiplicity $m_{\lambda}$ and let $\gamma\coloneqq \frac{m_{\lambda}}{m}$ (note that $\deg(\mathfrak{p}) = mb$ and that since $\mathfrak{p} = (x^{-\sigma/\rho}y)^{b'} p$, the multiplicity of $\lambda$ as a root of $p$ is also $m_{\lambda}$). By Proposition~\ref{multiplicidad}(3)
$$
\gamma\le \frac{b-b'}{\gap(\rho,l)} \le b.
$$
Hence, if $b = \gamma$, then $b'=0$, $\gap(\rho,l) = 1$ and $\mathfrak{p}(z) = \mu(z-\lambda)^{mb}$, and consequently $(A,(\rho,\sigma))$ is not of type~II. Since $mb>1$ it follows from Remark~\ref{comentarios0} that it is not of type~I either, and so it is necessarily of type~III. In line~7 of Algorithm~\ref{GetGeneratedCorners} we set $\gmax\coloneqq \min\left\{\frac{b-b'}{\gap(\rho,l)},b-1\right\}$ in order to avoid the regular corners of type~III. We can ignore these corners, since they do not appear in a complete chain of an $(m,n)$-pair (see Proposition~\ref{standard pair generates complete chain}). Note that from $b'=0$ and $\gap(\rho,l) = 1$ it follows that $(\mathcal{A},\mathcal{A}')$ is not simple.
\end{remark}

\subsection{The children of a valid edge}\label{the children}

%Let $(A,(\rho,\sigma))$ be a regular corner of type~II of an $(m,n)$-pair $(P,Q)$ and let $A'\coloneqq \frac{1}{m}\st_{\rho,\sigma}(P)$. By \cite{GGV1}*{Proposition~4.6(5)} there exists $(\rho_1,\sigma_1)\coloneqq \Pred_P(\rho,\sigma)$. If $(A,(\rho,\sigma))$ is of type~IIa, then we call $A_1\coloneqq A'$ the {\em corner generated by $(A,A')$} and we call $(A_1,A'_1)$, where $A'_1\coloneqq \frac 1m \st_{\rho_1,\sigma_1}(P)$, the {\em child of $(A,A')$}.

Let $(P,Q)$ be an $(m,n)$-pair in $L^{(l)}$, let $(A,(\rho,\sigma))$ be a regular corner of type~II of $(P,Q)$ and let $A'\coloneqq \frac{1}{m}\st_{\rho,\sigma}(P)$. If $(A,(\rho,\sigma))$ is of type~II.b), then applying~\cite{GGV1}*{Propositions~5.16 and~5.18(4)}, we obtain a regular corner $(A_1,(\rho',\sigma'))$ of an $(m,n)$-pair $(P_1,Q_1)$. In the sequel we will call $\mathcal{A}_1$ the corner generated by $(\mathcal{A},\mathcal{A}')$. If moreover $(A_1,(\rho',\sigma'))$ is of type~II, then we say that $(\mathcal{A}_1,\mathcal{A}_1')$, where $A'_1\coloneqq \frac{1}{m}\st_{\rho',\sigma'}(P_1)$, is a child of $(\mathcal{A},\mathcal{A}')$. On the other hand, if $(A,(\rho,\sigma))$ is of type~II.a), then we set $A_1\coloneqq A'$ and $A'_1\coloneqq \frac 1m \st_{\rho_1,\sigma_1}(P)$, where $(\rho_1,\sigma_1)\coloneqq \Pred_P(\rho,\sigma)$ (which is well defined by \cite{GGV1}*{Proposition~4.6(5)}). As before, in this case we also call $\mathcal{A}_1$ the corner generated by $(\mathcal{A},\mathcal{A}')$ and we say that $(\mathcal{A}_1,\mathcal{A}'_1)$ is a child of $(A,A')$.

\smallskip

For a general valid edge $(\mathcal{A},\mathcal{A}')$ we will construct all its possible children $(\mathcal{A}_1,\mathcal{A}_1')$ (see Definition~\ref{child}) in two steps:

\begin{itemize}[itemsep=0.5ex,topsep=0.5ex]

\item[-] {\bf GenerateCorners $(\mathcal{A},\mathcal{A}')$:} We find the corners $\mathcal{A}_1$ generated by a valid edge  $(\mathcal{A},\mathcal{A}')$ (see Definition~\ref{generated}).

\item[-] {\bf GetCornerChildren $((\mathcal{A},\mathcal{A}'),\mathcal{A}_1)$:} Given a corner $\mathcal{A}_1$ generated by a valid edge $(\mathcal{A},\mathcal{A}')$, we determine all possible $\mathcal{A}'_1$, such that $(\mathcal{A}_1,\mathcal{A}'_1)$ is a child of $(\mathcal{A},\mathcal{A}')$.

\end{itemize}

\smallskip

In the rest of this subsection $(\mathcal{A},\mathcal{A}')$ denotes a valid edge.

\begin{definition}
We set $\gamma_{\max}\coloneqq \min\bigl(\frac{b-b'}{\gap(\rho,l)},b-1\bigr)$ and we define the {\em set of multiplicities}
$$
\Gamma=\Gamma(\mathcal{A},\mathcal{A}')\coloneqq \begin{cases}\{\gamma_{\max}\}& \text{if $(\mathcal{A},\mathcal{A}')$ is simple}\\ \{b',\dots,\gamma_{\max}\}&\text{if $(\mathcal{A},\mathcal{A}')$ is not simple}.
\end{cases}
$$
\end{definition}

\begin{remark}\label{simple} Note that from the equality
$$
\gamma_{\max}=\min\bigl(\gcd(a-a',b-b'),b-1\bigr)
$$
(see \cite{GGV2}*{equality~(3.9)}) it follows that $\gamma_{\max}\in \mathds{N}$. Moreover if $\gamma_{\max} < \frac{b-b'}{\gap(\rho,l)}$, then $\gap(\rho,l) = 1$ and $b'= 0$, which, as we saw in Remark~\ref{comentarios}, excludes the case $(\mathcal{A},\mathcal{A}')$ simple.
\end{remark}

\begin{remark}
The previous definition is motivated by the properties established in Proposition~\ref{multiplicidad}(3) for the case of $(m,n)$-pairs.
\end{remark}

For each $\gamma$ such that  $b'\le \gamma \le \gamma_{\max}$, we let $\mathcal{A_{(\gamma)}}$ denote $(a_1\wrs l_1,b_1\bigr)$, where
$$
l_1\coloneqq \lcm(l,\rho),\quad b_1 \coloneqq \gamma\quad\text{and}\quad a_1\coloneqq \frac{al_1}{l} + (\gamma-b)\frac{-\sigma l_1}{\rho}.
$$
Note that $v_{\rho,\sigma}(A_{(\gamma)}) = v_{\rho,\sigma}(A)$. So $A_{(\gamma)}$ is in the line determined by $A$ and~$A'$.

\begin{definition}\label{admisible}
We say that $\mathcal{A_{(\gamma)}}$ is {\em admissible} if

\begin{enumerate}[itemsep=0.5ex,topsep=0.5ex]

\item $v_{1,-1}(A_{(\gamma)})<0$,

\item $l_1-\frac{a_1}{b_1}>1$ or $\gcd(a_1,b_1)>1$.

\end{enumerate}
\end{definition}

\begin{definition}\label{generated}
Let $\mathcal{A},\mathcal{A}'\in \mathds{N}_{\!(l)}\times \mathds{N}_0$ be such that $(\mathcal{A},\mathcal{A}')$ is a valid edge. We say that an element $\mathcal{A}_1\in \mathds{N}_{\!(l_1)}\times \mathds{N}$ is a {\em corner generated} by $(\mathcal{A},\mathcal{A}')$, if either $\mathcal{A}_1=\mathcal{A}'$ and $v_{1,-1}(A')<0$, or $v_{1,-1}(A')>0$ and there exists $\gamma\in \Gamma(\mathcal{A},\mathcal{A}')$ such that $\mathcal{A}_{(\gamma)}$ is admissible and $\mathcal{A}_1=\mathcal{A}_{(\gamma)}$ (which implies $\mathcal{A}_1\ne \mathcal{A}'$).
\end{definition}

\begin{proposition} Assume that $(\mathcal{A},\mathcal{A}')$ is simple. Let
$$
l_1\coloneqq \lcm(l,\rho),\quad a_1\coloneqq \frac{al_1}{l} + (\gamma_{\max}-b)\frac{-\sigma l_1}{\rho}\quad\text{and}\quad b_1\coloneqq \gamma_{\max}.
$$
If $v_{1,-1}(A')<0$, then $v_{1,-1}(A_1)>0$, where $A_1 \coloneqq \bigl(\frac{a_1}{l_1},b_1\bigr)$.
\end{proposition}

\begin{proof} By Definition~\ref{valid edges} and Remark~\ref{simple} we know that
\begin{equation}\label{eq7}
f_2=\gap(\rho,l)+1\qquad\text{and}\qquad \gmax=\frac{b-b'}{\gap(\rho,l)}.
\end{equation}
Let $\mu$ and $d$ be as in Definition~\ref{valid edges}. By Definition~\ref{valid edges} and item~(3') of Remark~\ref{zazaza} we have
\begin{equation}\label{eq6}
f_2 = \frac{\mu}{d} b\qquad\text{and}\qquad \frac{\mu}{d} = \frac{(\rho+\sigma)l}{\rho a + \sigma b l}.
\end{equation}
Moreover combining $v_{\rho,\sigma}(A)={v_{\rho,\sigma}(A')}$ with the fact that $v_{1,-1}(A')>0$, we  obtain
$$
b' < \frac {a'}{l}=-b'\frac{\sigma}{\rho}+\frac{a}l+b\frac{\sigma}{\rho}.
$$
Hence
$$
b'\left(\frac{\rho+\sigma}{\rho}\right)<\frac{\rho a+\sigma l b}{l\rho},
$$
which, by the second equality in~\eqref{eq6}, implies
$$
b'<\frac{\rho a+\sigma l b}{l(\rho+\sigma)}=\frac{d}{\mu}.
$$
But then, by the first equalities in~\eqref{eq7} and~\eqref{eq6},
$$
b=\frac{d}{\mu} f_2=\frac{d}{\mu}(\gap(\rho,l)+1)>\frac{d}{\mu} \gap(\rho,l) + b',
$$
and so, by the second equality in~\eqref{eq7},
$$
\gmax = \frac{b-b'}{\gap(\rho,l)}>\frac{d}{\mu}.
$$
Consequently,
$$
v_{1,-1}(A_1)=\frac{a\rho+b\sigma l}{\rho l}-\gmax \frac{\rho+\sigma}{\rho} < \frac{a\rho+b\sigma l}{\rho l}- \frac{d}{\mu}\frac{\rho+\sigma}{\rho}=0,
$$
where the last equality follows from the second equality in~\eqref{eq6}.
\end{proof}

In Algorithm~\ref{GetGeneratedCorners} we obtain a list $\GeneratedCorners$ consisting of all the corners generated by a valid edge $(\mathcal{A},\mathcal{A}')$.

\begin{center}
\begin{algorithm}[H]
%  \SetAlgoLined
\DontPrintSemicolon
%   \SetKwFunction{generateCorners}{generateCorners}
%   \SetKwProg{myalg}{Algorithm}{}{}
%   \SetKwProg{myproc}{Procedure}{}{}
%    \myproc{\generateCorners{$\mathcal{A},\mathcal{A}'$}}{
    \KwIn{A valid edge $(\mathcal{A},\mathcal{A}') = ((a\wrs l,b),(a'\wrs l, b'))$.}
	\KwOut{A list $\GeneratedCorners$, consisting of all generated corners by $(\mathcal{A},\mathcal{A}')$.}		
    %$\GeneratedCorners \gets \emptyset$\;
    $(\rho,\sigma)\gets \dir(A-A')$\;
	\eIf {$v_{1,-1}(A') < 0$}{add $\mathcal{A}'$ to $\GeneratedCorners$\;}{
        $l_1\gets \lcm(\rho,l)$\;
        $\gap\gets \frac{\rho}{\gcd(\rho,l)}$\;	
        $\gmax\gets \min\left\{\frac{b-b'}{\gap},b-1\right\}$\;
        \eIf {$\Simple(\mathcal{A},\mathcal{A}')=\TRUE$}{
                $a_1\gets \frac{al_1}{l} + (\gmax-b)\frac{-\sigma l_1}{\rho}$\;
				$\mathcal{A}_1\gets (a_1\wrs  l_1,\gmax)$\;
                \If {$l_1-a_1/b_1>1$ \normalfont\textbf{ or }
                         $\gcd(a_1,b_1)>1$}{add $\mathcal{A}_1$ to $\GeneratedCorners$\;}}
                {\For {$b_1\gets b'+1$ \KwTo $\gmax$}{
                    $a_1\gets \frac{al_1}{l} + (b_1-b)\frac{-\sigma l_1}{\rho}$\;
					$\mathcal{A}_1 \gets (a_1\wrs l_1,b_1)$\;
					\If { $v_{1,-1}(A_1)<0$\normalfont\textbf{ and } ($l_1-a_1/b_1>1$ \normalfont\textbf{ or }
                         $\gcd(a_1,b_1)>1$)}{add $\mathcal{A}_1$ to $\GeneratedCorners$\;}
					}
				}
		}
    {\bf RETURN} $\GeneratedCorners$
     \caption{GetGeneratedCorners}
     \label{GetGeneratedCorners}
\end{algorithm}
\end{center}

\begin{remark}\label{bala}
Definitions~\ref{admisible} and~\ref{generated} are motivated by the following fact: Let $(P,Q)$ be an $(m,n)$-pair in $L^{(l)}$ and let $(A,(\rho,\sigma))$ be a regular corner of type~II.b) of $(P,Q)$. Let $\varphi$ be the automorphism of $L^{(l_1)}$ introduced in \cite{GGV1}*{Proposition 5.18}, where $l_1\coloneqq \gcd(l,\rho)$. Let $\lambda\in K^{\times}$ be as in Proposition~\ref{multiplicidad}(4) and set
$$
A'\coloneqq \frac{1}{m}\st_{\rho,\sigma}(P),\quad A_1\coloneqq \frac{1}{m}\st_{\rho,\sigma}(\varphi(P)),\quad (\rho_1,\sigma_1)\coloneqq \Pred_{\varphi(P)}(\rho,\sigma)\quad\text{and}\quad \gamma\coloneqq \frac{m_{\lambda}}{m}.
$$
Then,

\begin{enumerate}[itemsep=0.5ex,topsep=0.5ex]

\item by Proposition~\ref{multiplicidad}(2) the pair $(\mathcal{A},\mathcal{A}')$ is a valid edge,

\item since $(A,(\rho,\sigma))$ is of type~II.b), we have $v_{1,-1}(A')>0$,

\item by~\cite{GGV1}*{Proposition~5.18(4)} the corner $\mathcal{A}_1$ satisfies condition~(1) of Definition~\ref{admisible},

\item by items~(3) and~(4) of Proposition~\ref{multiplicidad}, and Remark~\ref{comentarios}, we have $b'<\gamma\le\gamma_{\max}$. %    Moreover $b'<\gamma$, since otherwise $A_1 = \mathcal{A}_{(\gamma)} = A'_1$, which is impossible by items~(2) and~(3).

\item by~\cite{GGV1}*{Proposition~5.18(3)} we have $\mathcal{A}_{(\gamma)} = \mathcal{A}_1$,

\item by~\cite{GGV1}*{Proposition~5.19} the corner $\mathcal{A}_1$ satisfies condition~(2) of Definition~\ref{admisible}.

\end{enumerate}
Thus $\mathcal{A}_1\in \mathds{N}_{\!(l_1)}\times \mathds{N}$ is a corner generated by $(\mathcal{A},\mathcal{A}')$, $\mathcal{A}_1\ne \mathcal{A}'$ and there exists $b'  < \gamma\le \gamma_{\max}$ such that $\mathcal{A}_1 = \mathcal{A}_{(\gamma)}$, which implies that $v_{01}(A')< v_{01}(A_1)< v_{01}(A)$.
\end{remark}

\begin{definition}\label{child}
Let $(\mathcal{A},\mathcal{A}')$ and $(\mathcal{A}_1,\mathcal{A}_1')$ be valid edges and let $(\rho,\sigma)\coloneqq \dir(A - A')$ and $(\rho_1,\sigma_1)\coloneqq \dir(A_1 - A'_1)$.  We say that $(\mathcal{A}_1,\mathcal{A}_1')$ is a {\em child} of $(\mathcal{A},\mathcal{A}')$ if $(\rho,\sigma)>(\rho_1,\sigma_1)$ in $I$ and $\mathcal{A}_1$ is a corner generated by $(\mathcal{A},\mathcal{A}')$.
\end{definition}

The previous definition describes the main inductive construction that yields complete chains, generalizing the case when the valid edges correspond to an $(m,n)$-pair. This construction consists of the two steps mentioned above that are realized through Algorithms~\ref{GetGeneratedCorners} and \ref{GetCornerChildrenList}.

\smallskip

\begin{remark} Let $(\mathcal{A},\mathcal{A}')$ be a valid edge, let $(\rho,\sigma)\coloneqq \dir(A-A')$ and let $\mathcal{A}_1 = (a_1\wrs l_1,b_1)$ be a corned generated by $(\mathcal{A},\mathcal{A}')$. By Definition~\ref{generated} we know that $v_{1,-1}(A_1)<0$. In Algorithm~\ref{GetCornerChildrenList} we obtain all the children of $(\mathcal{A},\mathcal{A}')$ of the form $(\mathcal{A}_1,\mathcal{A}'_1)$. The lower bound $\lo$ in the algorithm comes from the fact that $(\rho_1,\sigma_1)<(\rho,\sigma)$ if and only if $\mu> \frac{d_1(\rho+\sigma)}{v_{\rho,\sigma}(A_1)}$, where $d_1\coloneqq \gcd(a_1,b_1)$. The upper bound $\hi$ in lines~4 and~6 and the conditions required in line~11 come from Definition~\ref{valid edges}. By~\cite{GGV2}*{Re\-mark~3.9} we know that
$$
A_1'=\Bigl(\frac{a_1}{l_1},b_1\Bigr)+j\Bigl(\gap(\rho_1,l_1) \frac{\sigma_1}{\rho_1},-\gap(\rho_1,l_1)\Bigr)\qquad\text{for some $0<j\le \Big\lfloor \frac{b_1}{\gap(\rho_1,l_1)}\Big\rfloor$.}
$$
\end{remark}

\begin{remark}
Before running Algorithm~\ref{GetCornerChildrenList} with input a corner $\mathcal{A}_1=(a_1\wrs l_1,b_1)$ such that $l_1-\frac{a_1}{b_1}\le 1$, and a valid edge$(\mathcal{A},\mathcal{A}')$, it is necessary to run Algorithm~\ref{PLLC} with input greater than or equal to $a_1$.
\end{remark}

\begin{center}
\begin{algorithm}[H]
%  \SetAlgoLined
\DontPrintSemicolon
%  \SetKwFunction{Directions}{Directions}
%  \SetKwProg{myalg}{Algorithm}{}{}
%  \SetKwProg{myproc}{Procedure}{}{}
%  \myproc{\Directions{$(\mathcal{A},\mathcal{A}'),\mathcal{A}_1$}}{
  \KwIn{A valid edge $(\mathcal{A},\mathcal{A}')$ and a corner $\mathcal{A}_1 = (a_1\wrs l_1,b_1)$ generated by $(\mathcal{A},\mathcal{A}')$ with $l_1-\frac{a_1}{b_1}\le 1$.}
  \KwOut{A list $\CornerChildrenList$, consisting of all $(\mathcal{A}_1,\mathcal{A}'_1)$ that are children of $(\mathcal{A},\mathcal{A}')$.}
    %$\ChildrenListI \gets \emptyset$\;
    $(\rho,\sigma) \gets \dir(A - A')$\;
    $d_1\gets \gcd(a_1,b_1)$\;
    $\lo \gets  \left\lfloor 1 + \frac{d_1(\rho + \sigma)}{v_{\rho,\sigma}(A_1)} \right\rfloor$\;
    $\hi \gets d_1$\;
    \If{$l_1>1$}{$\hi \gets \left\lfloor l_1(b_1l_1 - a_1) + \frac{d_1}{b_1} \right\rfloor$}
    \For {$\mu\gets \lo$ \KwTo $\hi$}{
        $\enF \gets \frac{\mu}{d_1}\left(\frac{a_1}{l_1}, b_1\right)$\;
        $(\rho_1, \sigma_1) \gets \dir(\enF-(1,1))$ \;
        $\gap \gets\frac{\rho_1}{\gcd(\rho_1,l_1)}$\;
	   \If {$\gap \leq b_1$ \normalfont\textbf{and} $d_1 \nmid \mu$ } {
		      \For {$j \gets 1$ \KwTo $\big\lfloor \frac{b_1}{\gap}\big\rfloor$}{
        $A'_1 \gets \bigl(\frac{a_1}{l_1},b_1\bigr)+j\bigl(\gap \frac{\sigma_1}{\rho_1},-\gap\bigr)$\;
        \If{( $l_1>1$\normalfont\textbf{ and } $v_{1,-1}(A'_1)\ne 0$ )\normalfont\textbf{ or }($l_1=1$\normalfont\textbf{ and } $v_{1,-1}(A'_1)< 0$)\normalfont\textbf{ or }\\
                \hspace{10pt}($l_1=1$, $v_{1,-1}(A'_1)> 0$\normalfont\textbf{ and } $A'_1\in \PLLC$)}{
                add $(\mathcal{A}_1,\mathcal{A}'_1)$ to $\CornerChildrenList$}
            }
		  }
        }
    {\bf RETURN} $\CornerChildrenList$
    \caption{GetCornerChildrenList}
\label{GetCornerChildrenList}
\end{algorithm}
\end{center}

\end{spacing}

\begin{spacing}{1.06}

\begin{definition}\label{final corner}
A corner $\mathcal{A}=(a\wrs l,b)$ is called a {\em final corner} if $l-\frac{a}{b}>1$.
\end{definition}

In Algorithm~\ref{GetChildrenAndFinalList} we combine Algorithms~\ref{GetGeneratedCorners} and~\ref{GetCornerChildrenList} in order to obtain a procedure giving the children of a valid edge $(\mathcal{A},\mathcal{A}')$ and the final corners generated by $(\mathcal{A},\mathcal{A}')$.

\smallskip

In line~$1$ of Algorithm~\ref{GetChildrenAndFinalList} we use the expression ``$\GetGeneratedCorners(\mathcal{A},\mathcal{A}')$'' as a notation for ``run $\GetGeneratedCorners$ with input $(\mathcal{A},\mathcal{A}')$''. We use similar notations in the following algorithms.

\begin{center}
\begin{algorithm}[H]
  %\SetAlgoLined
  \DontPrintSemicolon
  %\SetKwFunction{MainInductiveStep}{MainInductiveStep}
  %  \SetKwProg{myalg}{Algorithm}{}{}
  %  \SetKwProg{myproc}{Procedure}{}{}
  %  \myproc{\MainInductiveStep{$\mathcal{A},\mathcal{A}'$}}{
    \KwIn{A valid edge $(\mathcal{A},\mathcal{A}')$.}
	\KwOut{A list $\ChildrenList$, consisting of all children of $(\mathcal{A},\mathcal{A}')$.\\
            \hspace{1.45cm} A list $\FinalList$, consisting of all final corners generated by $(\mathcal{A},\mathcal{A}')$.}
    %$\ChildrenList \gets \emptyset$\;
    %$\FinalList \gets \emptyset$\;
    $\GeneratedCorners\gets \GetGeneratedCorners(\mathcal{A},\mathcal{A}')$\;
    \For{$\mathcal{A}_1=(a_1\wrs l_1,b_1)\in \GeneratedCorners$}{
      \If {$l_1-\frac{a_1}{b_1}>1$}{add $\mathcal{A}_1$ to $\FinalList$
				}$\CornerChildrenList\gets \GetCornerChildrenList((\mathcal{A},\mathcal{A}'),\mathcal{A}_1)$\;
        \For{ $(\mathcal{A}_1,\mathcal{A}'_1)\in \CornerChildrenList$ }{ add
            $(\mathcal{A}_1,\mathcal{A}'_1)$ to $\ChildrenList$
            }

		}
 {\bf RETURN} $(\ChildrenList,\FinalList)$
 \caption{GetChildrenAndFinalList}
\label{GetChildrenAndFinalList}
\end{algorithm}
\end{center}

\subsection{Main inductive step and complete chains}

Now we are able to construct recursively a chain $(\mathcal{C}_0,\dots,\mathcal{C}_j)$ of valid edges $\mathcal{C}_i\coloneqq (\mathcal{A}_i,\mathcal{A}'_i)$, where each $\mathcal{C}_i$ a child of the previous (except the first one). In the case of an standard $(m,n)$-pair $(P,Q)$, this process terminates when the generated corner
$$
\mathcal{A}_{j+1} = (a_{j+1}\wrs l_{j+1},b_{j+1})
$$
is a regular corner of type~I. In this case
$$
l_{j+1}-\frac{a_{j+1}}{b_{j+1}>1}.
$$

\begin{definition}\label{complete chain}
A chain $(\mathcal{C}_0,\dots,\mathcal{C}_j,\mathcal{A}_{j+1})$ is called a {\em complete chain of length~$j+1$}, if
\begin{itemize}[itemsep=0.5ex,topsep=0.5ex]

\item[-] $\mathcal{C}_i$ is a valid edge for $i=0,\dots,j$,

\item[-] $\mathcal{C}_{i+1}$ is a child of $\mathcal{C}_i$ for $i=0,\dots,j-1$,

\item[-] $\mathcal{A}_{j+1}$ is generated by $C_j$,

\item[-] $\mathcal{A}_{j+1}$ is a final corner,

\item[-] $l_0=1$,

\end{itemize}
where $\mathcal{C}_i=(\mathcal{A}_i,\mathcal{A}'_i)$ and $\mathcal{A}_i=(a_i\wrs l_i,b_i)$.
\end{definition}

In Algorithm~\ref{GetCompleteChains} we give a method for the generation of a list $\CompleteChains$ consisting of all complete chains starting with a valid edge
$$
\mathcal{C}_0 = (\mathcal{A},\mathcal{A}') = ((a,b),(a',b'))
$$
and having length less than or equal to $\NumberOfFactors\bigl(\gcd(b,(b-b')/\rho)\bigr)+1$, where $(\rho,\sigma)$ denotes $\dir(A-A')$ and $\NumberOfFactors \bigl(n)$ is an auxiliary function which returns the number of prime factors of $n$, counted with its multiplicity.

We use auxiliary lists $\OpenChains$ and $\POpenChains$ and an auxiliary variable $\Lmax$. Moreover the expression $\mathscr{C}\uplus \mathcal{A}_1$ denotes the chain obtained adding $\mathcal{A}_1$ at the end of the chain $\mathscr{C}$ and similarly for $\mathscr{C}\uplus (\mathcal{A}_1,\mathcal{A}'_1)$.

\end{spacing}

\begin{spacing}{1}

\begin{center}
\begin{algorithm}[H]
  %\SetAlgoLined
  \DontPrintSemicolon
  %\SetKwFunction{MainInductiveStep}{MainInductiveStep}
  %  \SetKwProg{myalg}{Algorithm}{}{}
  %  \SetKwProg{myproc}{Procedure}{}{}
  %  \myproc{\MainInductiveStep{$\mathcal{A},\mathcal{A}'$}}{
    \KwIn{A valid edge $\mathcal{C}_0=(\mathcal{A},\mathcal{A}') = ((a,b),(a',b'))$.}
	\KwOut{A list $\CompleteChains$, consisting of all complete chains $\mathcal{CH}$ starting in
    $\mathcal{C}_0$, with $\length(\mathcal{CH})\le \NumberOfFactors\Bigl(\gcd\Bigl(b,\frac{b-b'}{\rho}\Bigr)\Bigr)+1$, where $(\rho,\sigma)\coloneqq \dir(A-A')$.}
    $(\rho, \sigma) \gets \dir(A-A')$\;
    $\Lmax \gets \NumberOfFactors\left(\gcd\Bigl(b,\frac{b-b'}{\rho}\Bigr)\right)+1$\;
    $\OpenChains \gets (\mathcal{C}_0)$\;
    $j\gets 0$\;
    \While{$j < \Lmax $}{
    $\POpenChains \gets \emptyset$\;
    \For{$\mathcal{CH}\in \OpenChains$}{
    $\Last \gets $ \rm{Last element in} $\mathcal{CH}$\;
    $(\ChildrenList,\FinalList)\gets \GetChildrenAndFinalList(\Last)$\;
    \For{$\mathcal{A}_1\in \FinalList$}{
    \bf{add} $\mathcal{CH}\uplus \mathcal{A}_1$ \bf{to} $\CompleteChains$\;
         }
    \For{$(\mathcal{A}_1,\mathcal{A}'_1)\in \ChildrenList$}{
    \bf{add} $\mathcal{CH}\uplus (\mathcal{A}_1,\mathcal{A}'_1)$ \bf{to} $\POpenChains$\;
         }
    }
    $\OpenChains \gets \POpenChains$\;
    $j\gets j+1$}
 {\bf RETURN}  $\CompleteChains$
 \caption{GetCompleteChains}
\label{GetCompleteChains}
\end{algorithm}
\end{center}

\begin{theorem}\label{standard pair generates complete chain}
For each standard $(m,n)$-pair $(P,Q)$, there exist
$$
\bigl((P_i,Q_i),(A_i,A_i'),(\rho_i,\sigma_i),l_i\bigr)_{0\le i\le j}\quad\text{and}\quad \bigl((P_{j+1},Q_{j+1}),A_{j+1},(\rho_{j+1},\sigma_{j+1}),l_{j+1}\bigr)),
$$
where $j\in \mathds{N}$, such that:

\begin{enumerate}[itemsep=0.5675ex,topsep=0.5675ex]

\item $l_0\le \dots\le l_{j+1}\in \mathds{N}$ with $l_0 = 1$,

\item $(\rho_0,\sigma_0)>\dots > (\rho_{j+1},\sigma_{j+1})$ in $I$,

\item $(P_i,Q_i)$ is an $(m,n)$-pair in $L^{(l_i)}$ for each $1\le i\le j+1$ and $(P_0,Q_0) = (P,Q)$,

\item $\ell_{\rho_h,\sigma_h}(P_i) = \ell_{\rho_h,\sigma_h}(P_{i+1})$ for $0\le h<i\le j$,

\item $(A_h,(\rho_h,\sigma_h))$ is a regular corner of type~II.a) of $(P_i,Q_i)$ for $0\le h < i\le j+1$. Moreover
$$
\qquad\quad \frac 1m \st_{\rho_h,\sigma_h}(P_i) = A_{h+1}.
$$

\item $A_0 = \frac{1}{m}\en_{10}(P)$ and $(A_i,(\rho_i,\sigma_i))$ is a regular corner of type~II of $(P_i,Q_i)$ for $0\le i\le j$,

\item if $(A_i,(\rho_i,\sigma_i))$ is a regular corner of type~II.a) of $(P_i,Q_i)$, then
$$
\quad\qquad l_{i+1} = l_i,\quad (P_{i+1},Q_{i+1}) = (P_i,Q_i)\quad \text{and}\quad A_{i+1}=A'_i=\frac{1}{m}\st_{\rho_i,\sigma_i}(P_i),
$$

\item if $(A_i,(\rho_i,\sigma_i))$ is a regular corner of type~II.b) of $(P_i,Q_i)$, then $l_{i+1} = \lcm(\rho_i,l_i)$ and there exists a root $\lambda\in K^{\times}$ of the polynomial $\mathfrak{p}_i(z)$, defined by
    $$
    \ell_{\rho_i,\sigma_i}(P_i) = x^{\frac{k_i}{l_i}} \mathfrak{p}_i(z),\quad\text{where $z\coloneqq x^{-\sigma_i/\rho_i}y$,}
    $$
    such that $m\mid m_{\lambda}$, where $m_{\lambda}$ is the multiplicity of $z-\lambda$ in $\mathfrak{p}_i(z)$ and
    \begin{equation}\label{eq2}
    \frac{1}{m}\st_{\rho_i,\sigma_i}(P_{i+1}) = A_{i+1} = \Bigl(\frac{k_i}{m l_i},0\Bigr) + \frac{m_{\lambda}}{m}\Bigl(-\frac{\sigma_i}{\rho_i},1\Bigr)\ne A'_i=\frac{1}{m}\st_{\rho_i,\sigma_i}(P_i).
    \end{equation}
    Moreover $\ell_{\rho_i,\sigma_i}(P_{i+1}) = \varphi(\ell_{\rho_i,\sigma_i}(P_i))$, where $\varphi\in \Aut(L^{(l_{i+1})})$ is defined by
    $$
    \qquad\quad \varphi(x^{\frac{1}{l_{i+1}}})\coloneqq x^{\frac{1}{l_{i+1}}}\qquad\text{and}\qquad \varphi(y)\coloneqq y + \lambda x^{\frac{\sigma_i}{\rho_i}},
    $$

\item $(A_{j+1},(\rho_{j+1},\sigma_{j+1}))$ is a regular corner of type~I of $(P_{j+1},Q_{j+1})$ in $L^{(l_{j+1})}$,

\item $(\mathcal{A}_{i+1},\mathcal{A}_{i+1}')$ is a child of $(\mathcal{A}_i,\mathcal{A}_i')$ for $0\le i<j$,

\item $v_{01}(A_{i+1})<v_{01}(A_i)$ for $0\le i\le j$,

\item the chain
\begin{equation}\label{cadena completa}
\bigl((\mathcal{A}_0,\mathcal{A}_0'),\dots, (\mathcal{A}_j,\mathcal{A}_j'),\mathcal{A}_{j+1}\bigr),
\end{equation}
is complete,

\item if $t$ is the greatest index such that $l_t = 1$, then
\begin{itemize}[itemsep=0.565ex,topsep=0.565ex]
\item[-] $\bigl\{(A_i,(\rho_i,\sigma_i)): 0\le i\le t\bigr\}$ is the set of regular corners of $(P,Q)$,
\item[-] $(A_i,(\rho_i,\sigma_i))$ is a regular corner of type~IIa) of $(P,Q)$ for $0\le i<t$ and $(A_t,(\rho_t,\sigma_t))$ is a regular corner of type~IIb) of $(P,Q)$,
\item[-] $A'_t$ is the last lower corner of $(P,Q)$ (see \cite{GGV2}*{Definition~3.21}),
\item[-] $(P_i,Q_i) = (P,Q)$ for all $i\le t$,
\end{itemize}

\item The set of regular corners of $(P_{j+1},Q_{j+1})$ is $\{(A_i,(\rho_i,\sigma_i)):0\le i\le j+1\}$.

\end{enumerate}

\end{theorem}

\begin{proof}
Take the set
$$
\{(A_0,(\rho_0,\sigma_0)),\dots,(A_t,(\rho_t,\sigma_t))\},
$$
of regular corners of $(P,Q)$, with $(\rho_i,\sigma_i)>(\rho_{i+1},\sigma_{i+1})$ for all $i$ (note that we are using the opposed enumeration of~\cite{GGV1}*{Theorem~7.6}). By \cite{GGV1}*{Remark~5.12} we know that $A_0=\frac{1}{m}\en_{10}(P)$. Setting $A_i'\coloneqq \frac 1m \st_{\rho_i,\sigma_i}(P)$, we obtain a chain
$$
((A_0,A_0'),\dots,(A_t,A_t')),
$$
where $A_i,A_i'\in \mathds{N}\times\mathds{N}_0$ by \cite{GGV1}*{Remark~5.8}. By \cite{GGV1}*{Theorem~7.6(1)},
$$
\{(\rho_0,\sigma_0),\dots,(\rho_{t-1},\sigma_{t-1})\} = A(P)
$$
and the $3$-uple $(A_t,A'_t, (\rho_t,\sigma_t))$ is the starting triple of $(P,Q)$. Hence, by \cite{GGV1}*{Remark~5.10} we know that $(A_i,(\rho_i,\sigma_i))$ is a regular corner of type~II.a) of $(P,Q)$ for $0\le i<t$. Therefore $v_{1,-1}(A'_i)<0$ for $0\le i<t$. Furthermore, by items~(1) and~(2) of Proposition~\ref{multiplicidad} each one of the pairs $(\mathcal{A}_i,\mathcal{A}'_i)$, with $0\le i\le t$, is a valid edge. Moreover,
$$
A_{i+1} = A'_i\quad\text{and}\quad v_{01}(A_{i+1})<v_{01}(A_i)\qquad\text{for $0\le i<t$.}
$$
Consequently $\mathcal{A}_{i+1}$ is a corner generated by $(\mathcal{A}_i,\mathcal{A}'_i)$ for $0\le i <t$. Therefore $(\mathcal{A}_{i+1},\mathcal{A}'_{i+1})$ is a child of $(\mathcal{A}_i,\mathcal{A}'_i)$ for $0\le i <t$. Moreover, $A_t'$ is the last lower corner of $(P,Q)$. For $i\le t$, set $l_i\coloneqq 1$ and $(P_i,Q_i)\coloneqq (P,Q)$. By \cite{GGV1}*{Remark~6.3} we know that $(A_t,(\rho_t,\sigma_t))$ is a regular corner of type~II.b), and so $v_{1,-1}(\st_{\rho_t,\sigma_t}(P))>0$. This implies that $(\rho_t,\sigma_t)\ne (1,0)$, because $(P,Q)$ is standard (see \cite{GGV1}*{Definition~4.3}). Since $(\rho_t,\sigma_t)\in I$ we obtain that $\rho_t>0$. Let $\lambda\in K^{\times}$ be as in Proposition~\ref{multiplicidad}(4) and let \hbox{$l_{t+1}\coloneqq \rho_t$}. Applying \cite{GGV1}*{Proposition~5.18 and Remark~3.9} to $(P_t,Q_t)$ and $(A_t,(\rho_t,\sigma_t))$, we obtain an $(m,n)$-pair $(P_{t+1},Q_{t+1})$ in $L^{(l_{t+1})}$, such that

\begin{itemize}[itemsep=0.5675ex,topsep=0.5675ex]

\item[-] $\en_{\rho_t,\sigma_t}(P_{t+1}) = \en_{\rho_t,\sigma_t}(P_t)$ and $\ell_{\rho_h,\sigma_h}(P_{t+1}) = \ell_{\rho_h,\sigma_h}(P_t)$ for $0\le h < t$,

\item[-] $(A_{t+1},(\rho_{t+1},\sigma_{t+1}))$ is a regular corner of $(P_{t+1},Q_{t+1})$, where
$$
\qquad\quad (\rho_{t+1},\sigma_{t+1})\coloneqq \Pred_{P_{t+1}}(\rho_t,\sigma_t)\quad\text{and}\quad A_{t+1}\coloneqq \frac{1}{m}\st_{\rho_t,\sigma_t}(P_{t+1}),
$$

\item[-] There exists $\lambda\in K^{\times}$ such that $m$ divides the multiplicity $m_{\lambda}$ of $z-\lambda$ in $\mathfrak{p}_t(z)$ and
$$
\qquad\quad A_{t+1} = \Bigl(\frac{k_t}{m l_t},0\Bigr) + \frac{m_{\lambda}}{m} \Bigl(-\frac{\sigma_t}{\rho_t}, 1\Bigr),
$$
Moreover $\ell_{\rho_t,\sigma_t}(P_{t+1}) = \varphi(\ell_{\rho_t,\sigma_t}(P_t))$, where $\varphi\in \Aut(L^{(l_{t+1})})$ is defined by
$$
\qquad\quad \varphi(x^{\frac{1}{l_{t+1}}})\coloneqq x^{\frac{1}{l_{t+1}}}\qquad\text{and}\qquad \varphi(y)\coloneqq y+\lambda x^{\frac{\sigma_t}{\rho_t}},
$$

\item[-] $A(P_{t+1}) = A(P_t)\cup \{(\rho_t,\sigma_t)\}\cup \{(\rho,\sigma)\in A(P_{t+1}):(\rho,\sigma)<(\rho_t,\sigma_t)\text{ in } I\}$, where $A(P_t)$ and $A(P_{t+1})$ are as in the discussion above \cite{GGV1}*{Proposition~5.2}.

\end{itemize}
By Remark~\ref{bala} we know that $\mathcal{A}_{t+1}$ is a corner generated by $(\mathcal{A}_t,\mathcal{A}'_t)$, that $\mathcal{A}_{t+1}\ne \mathcal{A}'_t$ and that $v_{01}(A_{t+1})<v_{01}(A_t)$. We claim that we can assume that $(A_{t+1},(\rho_{t+1},\sigma_{t+1}))$ is of type~I or~II. In fact, suppose that it is a regular corner of type~III and write
$$
\ell_{\rho_{t+1},\sigma_{t+1}}(P_{t+1}) = x^{\frac{\kappa_{t+1}}{l_{t+1}}}\mu_0 (z-\lambda_0)^{r_0}\quad\text{where $z\coloneqq x^{\frac{-\sigma_{t+1}}{\rho_{t+1}}}y$, $\mu_0,\lambda_0\in K^{\times}$ and $r_0\in \mathds{N}$.}
$$
Then, by \cite{GGV1}*{Theorem~7.6(1) and Remark~5.10},
$$
A(P_{t+1}) = A(P_t)\cup \{(\rho_t,\sigma_t)\}
$$
while, by~\cite{GGV1}*{Proposition~5.17}, we have $\rho_{t+1}\mid l_{t+1}$ and there exists an $(m,n)$-pair $(P_{t+1,1},Q_{t+1,1})$ in $L^{(l_{t+1})}$ such that,

\begin{itemize}[itemsep=0.5675ex,topsep=0.5675ex]

\item[-] $\en_{\rho_{t+1},\sigma_{t+1}}(P_{t+1,1}) = \en_{\rho_{t+1},\sigma_{t+1}}(P_{t+1}) = A_{t+1} = \frac{1}{m}\st_{\rho_{t+1},\sigma_{t+1}}(P_{t+1,1})$,

\item[-]  $\ell_{\rho_h,\sigma_h}(P_{t+1,1}) = \ell_{\rho_h,\sigma_h}(P_{t+1})$ for $0\le h\le t$,

\item[-] $(A_{t+1},(\rho_{t+1,1},\sigma_{t+1,1}))$ is a regular corner of $(P_{t+1,1},Q_{t+1,1})$, where
$$
(\rho_{t+1,1},\sigma_{t+1,1})\coloneqq \Pred_{P_{t+1,1}} (\rho_{t+1},\sigma_{t+1}),
$$

\item[-] $A(P_{t+1,1}) = A(P_{t+1})\cup\{(\rho,\sigma)\in A(P_{t+1,1}):(\rho,\sigma)<(\rho_{t+1},\sigma_{t+1}) \text{ in } I\}$.

\end{itemize}
Note that $(\rho_{t+1,1},\sigma_{t+1,1}) = \Pred_{P_{t+1,1}} (\rho_t,\sigma_t)$. As long as Case~III occurs, we can find
$$
(\rho_{t+1,1},\sigma_{t+1,1})> \dots> (\rho_{t+1,u},\sigma_{t+1,u})>\dots,
$$
and $(m,n)$-pairs $(P_{t+1,u},Q_{t+1,u})$ in $L^{(l_{t+1})}$ such that for all $u\ge 1$

\begin{itemize}[itemsep=0.5675ex,topsep=0.5675ex]

\item[-] $\rho_{t+1,u}\mid l_{t+1}$ ,

\item[-] $\en_{\rho_{t+1,u},\sigma_{t+1,u}}(P_{t+1,u+1}) = \en_{\rho_{t+1,u},\sigma_{t+1,u}} (P_{t+1,u}) = A_{t+1} = \frac{1}{m}\st_{\rho_{t+1,u},\sigma_{t+1,u}}(P_{t+1,u+1})$,

\item[-] $(A_{t+1},(\rho_{t+1,u+1},\sigma_{t+1,u+1}))$ is a regular corner of $(P_{t+1,u+1},Q_{t+1,u+1})$, where
$$
(\rho_{t+1,u+1},\sigma_{t+1,u+1})\coloneqq \Pred_{P_{t+1,u+1}} (\rho_{t+1,u},\sigma_{t+1,u}) = \Pred_{P_{t+1,u+1}} (\rho_t,\sigma_t),
$$

\item[-] $\ell_{\rho_h,\sigma_h}(P_{t+1,u+1}) = \ell_{\rho_h,\sigma_h}(P_{t+1,u})$ for $0\le h\le t$,

\item[-] $A(P_{t+1,u+1}) =  A(P_{t+1}) \cup \{(\rho,\sigma)\in A(P_{t+1,u+1}):(\rho,\sigma)< (\rho_{t+1}, \sigma_{t+1}) \text{ in } I\}$.

\end{itemize}
But there are only finitely many $\rho_{t+1,u}$'s with $\rho_{t+1,u}\mid l_{t+1}$. Moreover,
$$
0<-\sigma_{t+1,u}<\rho_{t+1,u},
$$
since $(1,-1)<(\rho_{t+1,u},\sigma_{t+1,u})<(1,0)$, and so there are only finitely many $(\rho_{t+1,u},\sigma_{t+1,u})$ possible. Thus, eventually cases~I or~II must occur, proving the claim. Note that by \cite{GGV1}*{Theorem~7.6(1) and Remarks~5.10 and~5.11}
$$
(A_{t+1},(\rho_{t+1},\sigma_{t+1}))\text{ is of type~II.a)} \Leftrightarrow (\rho_{t+1},\sigma_{t+1})\in A(P_{t+1}) \Leftrightarrow A(P_t)\cup \{(\rho_t,\sigma_t)\}\subsetneq A(P_{t+1}).
$$
Assume that $(A_{t+1},(\rho_{t+1},\sigma_{t+1}))$ is a regular corner of type~II and set $A_{t+1}'\coloneqq \frac{1}{m}\st_{\rho_{t+1},\sigma_{t+1}}(P_{t+1})$. By Pro\-position~\ref{multiplicidad}(2) we know that $(\mathcal{A}_{t+1},\mathcal{A}_{t+1}')$ is a child of $(\mathcal{A}_t,\mathcal{A}_t')$. If $(A_{t+1},(\rho_{t+1},\sigma_{t+1}))$ is a regular corner of type~II.a), then by \cite{GGV1}*{Remark 5.11}, the pair
\begin{equation*}
\bigl(A_{t+2},(\rho_{t+2},\sigma_{t+2})\bigr)\coloneqq \bigl(A_{t+1}',\Pred_{P_{t+1}}(\rho_{t+1},\sigma_{t+1})\bigr)
\end{equation*}
is a regular corner of $(P_{t+2},Q_{t+2})\coloneqq (P_{t+1},Q_{t+1})$. Moreover, by definition $\mathcal{A}_{t+2}$ is generated by $(\mathcal{A}_{t+1},\mathcal{A}_{t+1}')$ and $v_{01}(A_{t+2})<v_{01}(A_{t+1})$. On the other hand, if $(A_{t+1},(\rho_{t+1},\sigma_{t+1}))$ is a cor\-ner of type~II.b), then, arguing as above we obtain a root $\lambda$ of $\mathfrak{p}_{t+1}(z)$ and an $(m,n)$-pair $(P_{t+2},Q_{t+2})$ in $L^{(l_{t+2})}$, where $l_{t+2}\coloneqq \lcm(l_{t+1},\rho_{t+1})$, such that

\begin{itemize}[itemsep=0.5675ex,topsep=0.5675ex]

\item[-] $\en_{\rho_{t+1},\sigma_{t+1}}(P_{t+2}) = \en_{\rho_{t+1},\sigma_{t+1}}(P_{t+1})$ and  $\ell_{\rho_h,\sigma_h}(P_{t+2}) = \ell_{\rho_h,\sigma_h}(P_{t+1})$ for $0\le h < t+1$,

\item[-] $(A_{t+2},(\rho_{t+2},\sigma_{t+2}))$ is a regular corner of type~I or~II of $(P_{t+2},Q_{t+2})$, where
$$
\qquad\quad (\rho_{t+2},\sigma_{t+2})\coloneqq \Pred_{P_{t+2}}(\rho_{t+1},\sigma_{t+1})\quad\text{and}\quad  A_{t+2}\coloneqq \frac{1}{m} \st_{\rho_{t+1},\sigma_{t+1}}(P_{t+2}),
$$

\item[-] $\mathcal{A}_{t+2}\ne \mathcal{A}_{t+1}'$, the pair $(\mathcal{A}_{t+1},\mathcal{A}_{t+1}')$ generates $\mathcal{A}_{t+2}$, and $v_{01}(A_{t+2})<v_{01}(A_{t+1})$,

\item[-] there exists $\lambda\in K^{\times}$ such that $m$ divides the multiplicity $m_{\lambda}$ of $z-\lambda$ in $\mathfrak{p}_{t+1}(z)$ and
$$
\qquad\quad A_{t+2} = \Bigl(\frac{k_{t+1}}{m l_{t+1}},0\Bigr) + \frac{m_{\lambda}}{m} \Bigl(-\frac{\sigma_{t+1}}{\rho_{t+1}}, 1\Bigr).
$$
Moreover $\ell_{\rho_{t+1},\sigma_{t+1}}(P_{t+2}) = \varphi(\ell_{\rho_{t+1},\sigma_{t+1}}(P_{t+1}))$, where $\varphi\in \Aut(L^{(l_{t+2})})$ is defined by
$$
\qquad\quad \varphi(x^{\frac{1}{l_{t+2}}})\coloneqq x^{\frac{1}{l_{t+2}}}\qquad\text{and}\qquad \varphi(y)\coloneqq y+\lambda x^{\frac{\sigma_{t+1}}{\rho_{t+1}}},
$$

\item[-] $A(P_{t+2}) = A(P_{t+1})\cup \{(\rho_{t+1},\sigma_{t+1})\}\cup \{(\rho,\sigma)\text{ in } A(P_{t+1}): (\rho,\sigma)<(\rho_{t+1},\sigma_{t+1})\in I\}$.

\end{itemize}
While regular corners of type~II occurs we continue with this process. Eventually a regular corner $(A_{j+1},(\rho_{j+1},\sigma_{j+1}))$ of type~I must occur. Finally, by \cite{GGV1}*{Proposition~5.13}, the chain~\eqref{cadena completa} is complete.
\end{proof}

\begin{remark}\label{rem 2.21p} By Theorem~\ref{modif final} below, if $(A_{j+1},(\rho_{j+1},\sigma_{j+1}))$ is a regular corner of type~I.a) of $(P_{j+1},Q_{j+1})$ in $L^{(l_{j+1})}$, then we can modify $(P_{j+1},Q_{j+1})$ in such a way that $(A_{j+1},(\rho_{j+1},\sigma_{j+1}))$ becomes of type~I.b).
\end{remark}

\begin{remark}\label{rem 2.21} Let $(P,Q)$ be a standard $(m,n)$-pair, let $j\in \mathds{N}$ and let
$$
\bigl((P_i,Q_i),(A_i,A_i'),(\rho_i,\sigma_i),l_i\bigr)_{0\le i\le j}\quad\text{and}\quad \bigl((P_{j+1},Q_{j+1}),A_{j+1},(\rho_{j+1},\sigma_{j+1}),l_{j+1}\bigr)
$$
satisfying items~(1)--(14) of Theorem~\ref{standard pair generates complete chain}. Let $h$ and $i$ be integers with $0\le h\le i\le j$. By items~(3), (5) and~(6), and \cite{GGV1}*{Theorem~7.6(2)}, there exists $d_h^{(i)}$ maximum such that
\begin{equation}\label{eq1}
\ell_{\rho_h,\sigma_h}(P_i) = R_{hi}^{m d_h^{(i)}}\qquad\text{for some $(\rho_h,\sigma_h)$-homogeneous $R_{hi}\in L^{(l_i)}$.}
\end{equation}
By item~(8) of~\cite{GGV1}*{Theorem~7.6} we know that
\begin{equation}\label{eqp5}
\# \Primefactors (d_h^{(i)})\ge i-h.
\end{equation}
Write $A_h = (a_h/l_h,b_h)$, $A_{h+1} = (a_{h+1}/l_{h+1},b_{h+1})$ and $A'_h = (a'_h/l_h,b'_h)$. We assert that
\vspace{1pt}
\begin{equation}\label{eq5}
d_h^{(i)}\Big| D_h^{(i)}\coloneqq \gcd\left(\frac{b_h-b'_h}{\gap(\rho_h,l_h)},b_h,b_{h+1},\frac{a_h l_i}{l_h},\frac{a'_h l_i}{l_h}\right).
\vspace{1pt}
\end{equation}
First note that by Theorem~\ref{standard pair generates complete chain}(5)
$$
(a_{h+1}/l_{h+1},b_{h+1}) = A_{h+1} = \frac 1m \st_{\rho_h,\sigma_h}(P_i)=d^{(i)}_h\st_{\rho_h,\sigma_h}(R_{hi}),
$$
and consequently $d_h^{(i)}|b_{h+1}$. By items~(4), (7) and~(8) of Theorem~\ref{standard pair generates complete chain} there exists $\lambda\in K$ such that
$$
\ell_{\rho_h,\sigma_h}(P_i) = \ell_{\rho_h,\sigma_h}(P_{h+1}) = \varphi(\ell_{\rho_h,\sigma_h}(P_h)),
$$
where $\varphi\in \Aut(L^{(l_{h+1})})$ is defined by
$$
\varphi(x^{\frac{1}{l_{h+1}}})\coloneqq x^{\frac{1}{l_{h+1}}}\qquad\text{and}\qquad \varphi(y)\coloneqq y + \lambda x^{\frac{\sigma_h}{\rho_h}}.
$$
Write $\widetilde{R}_{hi}\coloneqq \varphi^{-1}(R_{hi})$. Then
$$
\ell_{\rho_h,\sigma_h}(P_h) = \varphi^{-1}(\ell_{\rho_h,\sigma_h}(P_i)) = \widetilde R_{hi}^{m d_h^{(i)}},
$$
and so
$$
(A_h,A'_h)=((a_h/l_h,b_h),(a'_h/l_h,b'_h))=\Bigl(\en_{\rho_h,\sigma_h}\Bigl(\widetilde R_{hi}^{d_h^{(i)}}\Bigr),\st_{\rho_h,\sigma_h}\Bigl(\widetilde R_{hi}^{d_h^{(i)}}\Bigr)\Bigr).
$$
(Note that $\lambda = 0$ if and only if $(A_h,(\rho_h,\sigma_h))$ is a regular corner of type~II.a) of $(P_h,Q_h)$). Set  $z\coloneqq x^{-\frac{\sigma_h}{\rho_h}}y$ and write
$$
\widetilde R_{hi}^{d_h^{(i)}} = x^{\frac{a'_h}{l_h}}y^{b'_h}f_{hi}(z)\quad\text{and}\quad \widetilde R_{hi} = x^{\frac{u'_h}{l_i}} y^{v'_h} g_{hi}(z),
$$
where $f_{hi}$ and $g_{hi}$ are polynomials such that $f_{hi}(0)\ne 0$ and $g_{hi}(0)\ne 0$. Clearly
\vspace{1pt}
\begin{equation}\label{eq3}
d_h^{(i)}\Bigl| b'_h,\quad d_h^{(i)}\Bigl| b_h,\quad d_h^{(i)}\Bigl| \frac{a'_hl_i}{l_h},\quad d_h^{(i)}\Bigl| \frac{a_hl_i}{l_h}\quad \text{and}\quad f_{hi} = g_{hi}^{d_h^{(i)}}.
\vspace{1pt}
\end{equation}
Thus $d_h^{(i)}$ divides $b_h-b'_h$. We next prove that
\begin{equation}\label{eq4}
d_h^{(i)}\Bigl|\frac{b_h-b'_h}{\gap(\rho_h,l_h)}.
\end{equation}
Assume for a moment that $\gap(\rho_h,l_h)\mid t_{hi}$ where $t_{hi}\coloneqq \deg g_{hi}$ and write $t_{hi} = \gap(\rho_h,l_h) t'_{hi}$. From
$$
x^{\frac{a_h-a'_h}{l_h}}y^{b_h-b'_h} = z^{t_{hi}d_h^{(i)}} = x^{-\frac{t_{hi}d_h^{(i)}\sigma_h}{\rho_h}} y^{\gap(\rho_h,l_h) t'_{hi} d_h^{(i)}},
$$
we obtain that
$$
\gap(\rho_h,l_h) d_h^{(i)}\mid b_h-b'_h,
$$
from which~\eqref{eq4} follows. Consequently, we are reduced to prove that $\gap(\rho_h,l_h)\mid t_{hi}$. Suppose this is false and write
$$
g_{hi} = \sum_{u=0}^{t_{hi}} a_uz^u
$$
Let $v$ be the minimum $u$ such that $a_u\ne 0$ and $\gap(\rho_h,l_h)\nmid u$. A direct computation using that $\gap(\rho_h,l_h)\nmid v$ and that $\gap(\rho_h,l_h)\mid u$ for all $u<v$ such that $a_u\ne 0$, shows that the coefficient of $z^v$ in $g_{hi}^{md_h^{(i)}}(z)$ is $md_h^{(i)}a_0^{md_h^{(i)}-1}a_v\ne 0$. But this is impossible, since
$$
x^{\frac{ma'_h}{l_h}}y^{mb'_h} g_{hi}^{md_h^{(i)}}(z)
=\widetilde R_{hi}^{md_h^{(i)}} = \ell_{\rho_h,\sigma_h}(P_h)\in L^{(l_h)}\qquad\text{and}\qquad z^v = x^{-\frac{\sigma_h v}{\rho_h}} y^v\notin L^{(l_h)}.
$$
This proves~\eqref{eq4} and thus finishes the proof of~\eqref{eq5}.
\end{remark}

\begin{remark} From inequality~\eqref{eqp5} and condition~\eqref{eq5} (both with $h=0$ and $i=j$), we obtain that $j\le \# \Primefactors(D)$, where $D\coloneqq \gcd(b_0,(b_0-b'_0)/\rho_0)$.
\end{remark}

\subsection{Divisibility conditions and admissible complete chains}

In this subsection we first prove that if a complete chain $\mathscr{C} = (\mathcal{C}_0,\dots,\mathcal{C}_j,\mathcal{A}_{j+1})$ is constructed from a standard $(m,n)$-pair $(P,Q)$ as in Theorem~\ref{standard pair generates complete chain}, then $\mathscr{C}$ satisfies certain arithmetic conditions. In Definition~\ref{cond div} we name arbitrary complete chains that satisfy these properties ``admissible complete chains''. Then we obtain a procedure in order to determine if a given complete chain is admissible.

\smallskip

Let $(P,Q)$ be an standard $(m,n)$-pair, let $j\in \mathds{N}$ and let
$$
\bigl((P_i,Q_i),(A_i,A_i'),(\rho_i,\sigma_i),l_i\bigr)_{0\le i\le j}\quad\text{and}\quad \bigl((P_{j+1},Q_{j+1}),A_{j+1},(\rho_{j+1},\sigma_{j+1}),l_{j+1}\bigr)),
$$
be as in Remark~\ref{rem 2.21}. By items~(3), (5) and~(6) of Theorem~\ref{standard pair generates complete chain} and \cite{GGV1}*{The\-orem~7.6(3)} (which applies since $v_{\rho_h,\sigma_h}(P_h)>0$ by \cite{GGV1}*{Corollary~5.7(1)}) for $h\le j$ there exist $p_h,q_h\in \mathds{N}$ coprime and a $(\rho_h,\sigma_h)$-homogeneous element $F_h\in L^{(l_h)}$ such that,
\begin{equation*}
v_{\rho_h,\sigma_h}(F_h) = \rho_h\!+\!\sigma_h,\quad [F_h,\ell_{\rho_h,\sigma_h}(P_h)] = \ell_{\rho_h,\sigma_h}(P_h) \quad\text{and}\quad \en_{\rho,\sigma}(F_h) = \frac{p_h}{q_h}\frac{1}{m}\en_{\rho_h,\sigma_h}(P_h).
\end{equation*}
Let $\varphi\in \Aut(L^{(l_{h+1})})$ be as in Remark~\ref{rem 2.21}. Since $\varphi$ is $(\rho_h,\sigma_h)$-homogeneous,
$$
v_{\rho_h,\sigma_h}(\varphi(F_h)) = \rho_h+\sigma_h.
$$
Moreover, by \cite{GGV1}*{Remark~3.10} and items~(7) and~(8) of Theorem~\ref{standard pair generates complete chain},
$$
[\varphi(F_h),\ell_{\rho_h,\sigma_h}(P_{h+1})] = [\varphi(F_h),\varphi(\ell_{\rho_h,\sigma_h}(P_h))] = \varphi(\ell_{\rho_h,\sigma_h}(P_h)) = \ell_{\rho_h,\sigma_h}(P_{h+1}).
$$
Thus, by item~(4) of Theorem~\ref{standard pair generates complete chain}
\begin{equation}\label{eq17}
[\varphi(F_h),\ell_{\rho_h,\sigma_h}(P_i)] = \ell_{\rho_h,\sigma_h}(P_i)\qquad\text{for $h<i\le j$.}
\end{equation}
Since $\rho_h>0$, the end point of each $(\rho_h,\sigma_h)$-homogeneous element $F$ of $L^{(l_i)}$ is the support of the monomial of greatest degree in $y$ of $F$. Consequently
$$
\en_{\rho_h,\sigma_h}(F_h) = \en_{\rho_h,\sigma_h}(\varphi(F_h)),
$$
because the monomials of greatest degree in $y$ of $F_h$ and $\varphi(F_h)$ coincide. Note that since $(A_h,(\rho_h,\sigma_h))$ is a regular corner of type~II) of $P_i$ the hypothesis of \cite{GGV1}*{Proposition~2.11(5)} are fulfilled, and so $\varphi(F_h)$ is the unique $(\rho_h,\sigma_h)$-homogeneous element of $L^{(l_i)}$ that satisfies equality~\eqref{eq17}.

\begin{remark}\label{admis}
By items~(4), (5), (6) and~(8) of~\cite{GGV1}*{Theorem~7.6} the following conditions hold:

\begin{itemize}[itemsep=0.5ex,topsep=0.5ex]

\item[-] $q_h\nmid d_h^{(i)}$ for all $0\le h\le i\le  j$.

\item[-] $q_k \mid d_h^{(i)}$ for all $0\le h< k\le i\le  j$.

\item[-] $q_h\nmid q_k$ for all $0\le h<k\le  j$.

\end{itemize}
Note that since
$$
\gcd(p_h,q_h) = 1\quad\text{and}\quad \frac{p_h}{q_h}=\frac{\rho_h+\sigma_h}{v_{\rho_h,\sigma_h}(A_h)},
$$
we have
\begin{equation}\label{q_k}
p_h=\frac{\rho_h + \sigma_h}{\gcd(\rho_h+\sigma_h,v_{\rho_h,\sigma_h}(A_h))} \quad\text{and}\quad q_h=\frac{v_{\rho_h,\sigma_h}(A_h)}{\gcd(\rho_h+\sigma_h,v_{\rho_h,\sigma_h}(A_h))}.
\end{equation}
\end{remark}

Let $(\mathcal{C}_0,\dots,\mathcal{C}_j,\mathcal{A}_{j+1})$ be a complete chain (see Definition~\ref{complete chain}). For $0\le i\le j$, write
$$
\mathcal{C}_i = (\mathcal{A}_i,\mathcal{A}'_i),\quad \mathcal{A}_i = (a_i\wrs l_i,b_i),\quad \mathcal{A}'_i = (a'_i\wrs l_i,b'_i) \quad\text{and}\quad (\rho_i,\sigma_i)\coloneqq \dir(A_i-A'_i),
$$
and write
$$
\mathcal{A}_{j+1} = (a_{j+1}\wrs l_{j+1},b_{j+1}).
$$
Now for $0\le h\le j$, we can define $p_h$ and $q_h$ by equalities~\eqref{q_k}, and we do it. Moreover, as in Remark~\ref{rem 2.21}, we set
$$
D_h^{(i)}\coloneqq \gcd\left(\frac{b_h-b'_h}{\gap(\rho_h,l_h)},b_h,b_{h+1},\frac{a_h l_i}{l_h},\frac{a'_h l_i}{l_h}\right).
$$

\begin{definition}\label{cond div}
A complete chain is called an {\em admissible complete chain} if for all $0\le h<i\le  j$ it satisfies
$$
q_i \mid D_h^{(i)},\quad q_h\nmid q_i\quad\text{and}\quad \# \Primefactors (D_h^{(i)})\ge i-h.
$$
\end{definition}

By Remark~\ref{admis}, inequality~\eqref{eqp5} and condition~\eqref{eq5} every complete chains arising from a standard $(m,n)$-pair $(P,Q)$ is admissible. In Algorithm~\ref{GetIsAdmissible} we give a procedure to verify if an arbitrary complete chain is admissible.

\begin{center}
\begin{algorithm}[H]
  %\SetAlgoLined
  \DontPrintSemicolon
  %\SetKwFunction{MainInductiveStep}{MainInductiveStep}
  %  \SetKwProg{myalg}{Algorithm}{}{}
  %  \SetKwProg{myproc}{Procedure}{}{}
  %  \myproc{\MainInductiveStep{$\mathcal{A},\mathcal{A}'$}}{
    \KwIn{A complete chain $\mathcal{C} = (\mathcal{C}_0,\dots,\mathcal{C}_j,\mathcal{A}_{j+1})$ with $\mathcal{C}_i=(\mathcal{A}_i,\mathcal{A}'_i) = \bigl((a_i\wrs l_i,b_i),(a'_i\wrs l_i,b'_i)\bigr)$.}
	\KwOut{A boolean variable $\IsAdmissible$.}
    $h \gets 0$\;
    $i \gets 1$\;
    $\IsAdmissible\gets \TRUE$\;
    \While{$h<j$ \normalfont\textbf{ and } $\IsAdmissible = \TRUE$}{
    $(\rho, \sigma) \gets \dir(A_h-A'_h)$\;
    $\gap\gets \frac{\rho}{\gcd(\rho,l_h)}$\;
    $q\gets \frac{v_{\rho,\sigma}(A_h)}{\gcd(\rho+\sigma,v_{\rho,\sigma}(A_h))}$\;
       \While{$i\le j$ \normalfont\textbf{ and } $\IsAdmissible = \TRUE$}{
       $(\rho',\sigma') \gets \dir(A_i-A'_i)$\;
       $q'\gets \frac{v_{\rho',\sigma'}(A_i)}{\gcd(\rho'+\sigma',v_{\rho',\sigma'}(A_i))}$\;
       $D\gets \gcd\left(\frac{b_h-b'_h}{\gap},b_h,b_{h+1},\frac{a_h l_i}{l_h},\frac{a'_h l_i}{l_h}\right)$\;
        \eIf {$\# \Primefactors (D)\ge i-h$ \normalfont\textbf{ and } $q'\mid D$ \normalfont\textbf{ and } $q\nmid q'$}{$i\gets i+1$\;}{
            $\IsAdmissible\gets \FALSE$\;
				}
		}
    $h\gets h+1$\;
    $i\gets h+1$\;
    }
 {\bf RETURN} $\IsAdmissible$
 \caption{GetIsAdmissible}
\label{GetIsAdmissible}
\end{algorithm}
\end{center}

\end{spacing}

\begin{spacing}{0.99}

In Algorithm~\ref{Main algorithm} we obtain all admissible complete chains starting from a valid edge $(\mathcal{A},\mathcal{A}')$ with $v_{11}(A)\le M$ for a given upper bound $M$. Due to all the previous algorithms, this main procedure is short.

\begin{center}
\begin{algorithm}[H]
  %\SetAlgoLined
  \DontPrintSemicolon
  %\SetKwFunction{MainInductiveStep}{MainInductiveStep}
  %  \SetKwProg{myalg}{Algorithm}{}{}
  %  \SetKwProg{myproc}{Procedure}{}{}
  %  \myproc{\MainInductiveStep{$\mathcal{A},\mathcal{A}'$}}{
    \KwIn{A positive integer $M$.}
	\KwOut{A list $\AdmissibleCompleteChains$ of all admissible complete chains
    $(\mathcal{C}_0,\dots,\mathcal{C}_j,\mathcal{A}_{j+1})$, with $v_{11}(A_0)\le M$, where $\mathcal{A}_0$ is the first coordinate of $\mathcal{C}_0$.}
    $\PLLC\gets\GetPossibleLastLowerCorners\bigl(\bigl\lfloor \frac{M}{2}\bigr\rfloor\bigr)$\;
       \For{$a=2$ \KwTo $\bigl\lfloor \frac{M}{2}\bigr\rfloor$}{
    \For{$b=a+1$ \KwTo $M-a$}{
    $\StartingEdges \gets \GetStartingEdges((a,b),\PLLC)$\;
        \For{$(\mathcal{A},\mathcal{A}')\in \StartingEdges$ }{$\CompleteChains \gets \GetCompleteChains(\mathcal{A},\mathcal{A}')$\;
        \For{$\mathcal{CH}\in \CompleteChains$ }{$\IsAdmissible \gets \GetIsAdmissible(\mathcal{CH})$\;
        \If{$\IsAdmissible = \TRUE$}{add $\mathcal{CH}$ to $\AdmissibleCompleteChains$}
              }
            }
        }
    }
 {\bf RETURN} $\AdmissibleCompleteChains$
 \caption{Main algorithm}
\label{Main algorithm}
\end{algorithm}
\end{center}

We want to apply Algorithm~\ref{Main algorithm} in order to obtain limitations on the possible counterexamples to the Jacobian Conjecture. Assume then that this conjecture is false. By \cite{GGV1}*{Corollary~5.21} we know there exists a counterexample $(P,Q)$ and $m,n\in \mathds{N}$ coprime such that $(P,Q)$ is a standard $(m,n)$-pair and a minimal pair, which  means that $\gcd(v_{1,1}(P),v_{1,1}(Q))=B$, where $B$ is as in~\eqref{def B}.

\smallskip

Let $A_0$ be as in Remark~\ref{rem 2.21}. By \cite{GGV1}*{Proposition~5.2 and Corollary~5.21(3)}
$$
A_0 = \frac{1}{m}\en_{10}(P)  \quad\text{and}\quad \gcd(v_{11}(P),v_{11}(Q)) = v_{11}(A_0).
$$
By Theorem~\ref{standard pair generates complete chain} and Remark~\ref{admis} we know that $\mathcal{A}_0$ is the first coordinate of $\mathcal{C}_0$ for one of the admissible complete chains obtained running Algorithm~\ref{Main algorithm} with $M \ge B$.

\section[Generation of $(m,n)$-families parameterized by $\mathds{N}_0$]{Generation of $\bm{(m,n)}$-families parameterized by $\mathds{N}_0$}

In this section, for a complete chain $\mathscr{C}\coloneqq (\mathcal{C}_0,\dots,\mathcal{C}_j,\mathcal{A}_{j+1})$, we obtain restrictions on all the possible $m$ and $n$ such that there could exist an $(m,n)$-pair $(P,Q)$ that generates $\mathscr{C}$ as in Theo\-rem~\ref{standard pair generates complete chain}.

\begin{proposition}\label{modif final}
If an $(m,n)$-pair $(P,Q)$ in $L^{(l)}$ has a regular corner $(A,(\rho,\sigma))$ of type~I.a), then $\rho\mid l$ and there exists $\varphi\in\Aut(L^{(l)})$, such that $(\varphi(P),\varphi(Q))$ is an $(m,n)$-pair and $(A,(\rho,\sigma))$ is a regular corner of type~I.b) of $(\varphi(P),\varphi(Q))$. Moreover, the regular corners of $(P,Q)$ and the regular corners of $(\varphi(P),\varphi(Q))$, coincide.
\end{proposition}

\begin{proof} Let $A'\coloneqq \frac{1}{m}\st_{\rho,\sigma}(P)$ and write $A = (a/l,b)$ and $A' = (a'/l,b')$. By~\cite{GGV1}*{Proposition~5.13a)} we know that $b'=0$. Write
$$
\ell_{\rho,\sigma}(P) = x^{\frac{ma'}{l}}  p(z)\quad\text{with $z\coloneqq x^{-\frac{\sigma}{\rho}}y$, $p(z) = \sum a_iz^i\in K[z]$ and $a_0\ne 0$,}
$$
and
$$
\ell_{\rho,\sigma}(Q) = x^{\frac{na'}{l}}  q(z)\quad\text{with $z\coloneqq x^{-\frac{\sigma}{\rho}}y$, $q(z) = \sum b_iz^i\in K[z]$ and $b_0\ne 0$.}
$$
A direct computation shows that there exists $S\in L^{(l)}$, such that
$$
[\ell_{\rho,\sigma}(P),\ell_{\rho,\sigma}(Q)]= \frac{a'}{l}(ma_0b_1-na_1b_0) x^{\frac{ma'+na'}{l}-\frac{\sigma}{\rho}-1}+y S.
$$
Since $(A,(\rho,\sigma))$ of type~I, we have $[\ell_{\rho,\sigma}(P),\ell_{\rho,\sigma}(Q)]\ne 0$. So, by \cite{GGV1}*{Proposition~1.13}
\begin{equation}\label{eq8}
[\ell_{\rho,\sigma}(P),\ell_{\rho,\sigma}(Q)]=\ell_{\rho,\sigma}([P,Q])\in K^{\times}.
\end{equation}
Thus, necessarily $\frac{(m+n)a'}{l}-\frac{\sigma}{\rho}=1$ and $a_1\ne 0$ or $b_1\ne 0$. If $a_1\ne 0$, then
$$
\Bigl(\frac{ma'}{l}-\frac{\sigma}{\rho},1\Bigr)\in \Supp(\ell_{\rho,\sigma}(P))\subseteq \frac{1}{l}\mathds{Z}\times \mathds{N}_0.
$$
Since $\bigl(\frac{ma'}{l},0\bigr)$ also is in $\Supp(\ell_{\rho,\sigma}(P))\subseteq \frac{1}{l}\mathds{Z}\times \mathds{N}_0$, we conclude that $\frac{\sigma}{\rho}\in \frac{1}{l}\mathds{Z}$, which implies $\rho\mid l$. Si\-milarly, if $b_1\ne 0$, then we also obtain $\rho\mid l$, as desired. Now let $z-\lambda$ be a linear factor of $p(z)$. Define $\varphi\in \Aut(L^{(l)})$ by
$$
\varphi(x^{1/l})\coloneqq x^{1/l}\quad\text{and}\quad \varphi(y)\coloneqq y+\lambda x^{\sigma/\rho}.
$$
Then
$$
\varphi(\ell_{\rho,\sigma}(P)) = x^{\frac{ma'}{l}} p(z+\lambda) = x^{\frac{ma'}{l}}\ov p(z)\quad\text{and}\quad \varphi(\ell_{\rho,\sigma}(Q)) = x^{\frac{na'}{l}} q(z+\lambda) = x^{\frac{na'}{l}}\ov q(z),
$$
where $\ov p(z) = p(z+\lambda)$ and $\ov q(z) = q(z+\lambda)$. By~\cite{GGV1}*{Proposition~3.9} we know that, for all $H\in L^{(l)}$,
\begin{align}
&\ell_{\rho,\sigma}(\varphi(H)) = \varphi(\ell_{\rho,\sigma}(H)),\qquad \en_{\rho,\sigma}(\varphi(H)) = \en_{\rho,\sigma}(H)\notag
\shortintertext{and}
&\ell_{\rho_1,\sigma_1}(\varphi(H)) = \ell_{\rho_1,\sigma_1}(H)\quad\text{for all $(\rho,\sigma)< (\rho_1,\sigma_1) \le (1,1)$}.\label{eq9}
\end{align}
Using this with $H =P$ and $H = Q$, we obtain that
$$
\frac{v_{11}(\varphi(P))}{v_{11}(\varphi(Q))} = \frac{v_{10}(\varphi(P))}{v_{10}(\varphi(Q))} = \frac{m}{n}\quad\text{and}\quad v_{1,-1(\en_{10}(\varphi(P)))} < 0.
$$
Hence $(\varphi(P),\varphi(Q))$ is an $(m,n)$-pair, since $[\varphi(P),\varphi(Q)] = [P,Q]\in K^{\times}$, by~\cite{GGV1}*{Proposition~3.10}. We claim that $(\rho,\sigma)\in \Dir(\varphi(P))$. In fact since
$$
\ell_{\rho,\sigma}(\varphi(P)) = \varphi(\ell_{\rho,\sigma}(P)) = x^{\frac{ma'}{l}}\ov p(z),
$$
in order to see this it suffices to show that $\ov p$ is not a monomial, which follows easily from the fact that $\deg(p) = m(b-b')>1$ and $\lambda$ is a simple root of $p$ by Remark~\ref{comentarios0}. Write $\ov p(z)=\sum_i \ov a_iz^i$ and $\ov q(z)=\sum_i \ov b_iz^i$. By~\cite{GGV1}*{Proposition~3.10} and~\eqref{eq8}, we have
$$
[\ell_{\rho,\sigma}(\varphi (P)),\ell_{\rho,\sigma}(\varphi (Q))] = [\varphi(\ell_{\rho,\sigma}(P)),\varphi(\ell_{\rho,\sigma}(Q))] = \varphi([\ell_{\rho,\sigma}(P),\ell_{\rho,\sigma}(Q)]) \in K^{\times}.
$$
Using this and the fact that $\ov a_0 = p(\lambda) = 0$ we obtain
$$
-\frac{n a'}{l}\ov a_1\ov b_0 = [\ell_{\rho,\sigma}(\varphi (P)),\ell_{\rho,\sigma}(\varphi (Q))]\in K^{\times}.
$$
Hence
$$
\st_{\rho,\sigma}(\varphi(P))=\Bigl(\frac{ma'}{l}-\frac{\sigma}{\rho},1\Bigr)\quad\text{and}\quad \st_{\rho,\sigma}(\varphi(Q))=\Bigl(\frac{na'}{l},0\Bigr),
$$
and so $(A,(\rho,\sigma))$ is a regular corner of type~I.b) of $(\varphi(P),\varphi(Q))$. Using this, that $(A,(\rho,\sigma))$ is a regular corner of type~I) of $(P,Q)$, equalities~\eqref{eq9} with $H=P$, and~\cite{GGV1}*{Remark~5.10 and Theorem~7.6(1)}, we obtain that $(P,Q)$ and $(\varphi(P),\varphi(Q))$ have the same regular corners.
\end{proof}

Let $((a/l,b),(\rho,\sigma))$ be a regular corner of type~I.b) of an $(m,n)$-pair $(P,Q)$ in $L^{(l)}$. According to~\cite{GGV1}*{Proposition~5.13b)} there exists $k\in\mathds{N}$, with $k<l-\frac{a}{b}$ such that
\begin{equation}\label{conjunto de starting}
\{\st_{\rho,\sigma}(P),\st_{\rho,\sigma}(Q)\} = \left\{\left(\frac{k}{l},0\right), \left(1-\frac{k}{l},1\right) \right\},
\end{equation}

\begin{proposition}\label{extremosfinales1} Let $e_k\coloneqq  \gcd(k,bl-a)$. If $\st_{\rho,\sigma}(Q) = (k/l,0)$, then $\frac{k}{e_k}\mid n$ and
\begin{equation}\label{ecuacion diofantica}
(m+n)b-\frac{ne_k}{k} \frac{bl-a}{e_k}=1,
\end{equation}
while if $\st_{\rho,\sigma}(P)= (k/l,0)$, then $\frac{k}{e_k}\mid m$ and
\begin{equation}\label{ecuacion diofantica2}
(m+n)b-\frac{me_k}{k} \frac{bl-a}{e_k}=1.
\end{equation}
\end{proposition}

\begin{proof} Assume first that $\st_{\rho,\sigma}(Q) = (k/l,0)$. Since, by \cite{GGV1}*{Corollary~5.7(2)},
$$
\en_{\rho,\sigma}(P) = m\left(\frac{a}{l},b\right)\qquad\text{and}\qquad \en_{\rho,\sigma}(Q) = n\left(\frac{a}{l},b\right),
$$
we have
\begin{align*}
&\rho-\frac{\rho k}{l}+\sigma = v_{\rho,\sigma}(\st_{\rho,\sigma}(P)) = v_{\rho,\sigma}(\en_{\rho,\sigma}(P)) = m\left(\frac{a\rho}{l}+b\sigma\right)
\shortintertext{and}
&\frac{\rho k}{l}= v_{\rho,\sigma}(\st_{\rho,\sigma}(Q)) = v_{\rho,\sigma}(\en_{\rho,\sigma}(Q)) = n \left(\frac{a \rho}{l}+b\sigma\right),
\end{align*}
which leads to
\begin{equation}\label{igualdad st y en}
1-\frac k{l} +\frac{\sigma}{\rho} = \frac{ma}{l}+mb \frac{\sigma} {\rho}\qquad\text{and}\qquad \frac{\sigma}{\rho}=\frac{k-na}{nlb}.
\end{equation}
Hence,
$$
\frac{ma}{l}+mb\frac{k-na}{nlb}=1-\frac k{l} +\frac{k-na}{nl b},
$$
which gives
$$
(m+n)bk - n(bl-a) = k.
$$
Therefore $k\mid  n(bl-a)$. Since $\frac{k}{e_k}$ and $\frac{bl-a}{e_k}$ are coprime, necessarily $\frac{k}{e_k}\mid n$. So, equality~\eqref{ecuacion diofantica} is true. The case $\st_{\rho,\sigma}(P)= (k/l,0)$ is similar.
\end{proof}

Let $\mathcal{A}\coloneqq (a\wrs l,b)\in \mathds{N}_{\!(l)}\times \mathds{N}_0$ be a final corner and let $k\in \mathds{N}$ be such that $k < l - \frac{a}{b}$. We want to find all the $(m,n)\in \mathds{N}^2$ such that one of the equalities~\eqref{ecuacion diofantica} or~\eqref{ecuacion diofantica2} is satisfied. By symmetry it suffices to find the set of all those $(m,n)\in \mathds{N}^2$ such that equality~\eqref{ecuacion diofantica} is satisfied and then to add to this set the pairs obtained by swapping $m$ with $n$. For the first task we proceed as follows: we first check that
$$
\gcd\left(b,\frac{bl-a}{e_k}\right)=1,\quad\text{where $e_k\coloneqq \gcd(k,bl-a)$.}
$$
If this is the case we determine the Bezout coefficients $M,N$ with $N\ge 1$ in
$$
Mb-N\frac{bl-a}{e_k}=1.
$$
For each solution $(M,N)$ we set $n\coloneqq \frac{Nk}{e_k}$ and $m\coloneqq M-n$. Since $b<\frac{bl-a}{k}$, we have
$$
mb = Mb-\frac{Nk}{e_k}b > Mb-\frac{Nk}{e_k}\frac{bl-a}{k}=1,
$$
which implies that $m\ge 1$ as desired. Then we keep all the pairs $(m,n)$ that also satisfy $m>1$, $n>1$ and $\gcd(m,n) = 1$.

\begin{definition}\label{mn families}
Let $\mathcal{A}\coloneqq (a\wrs l,b)\in \mathds{N}_{\!(l)}\times \mathds{N}_0$ be a final corner and let
$$
I(\mathcal{A})\coloneqq \left\{k\in\mathds{N}: 1\le k<l-\frac{a}{b}\text{ and } \gcd\left(b,\frac{bl-a}{\gcd(k,bl-a)}\right) = 1\right\}.
$$
For each $k\in I(\mathcal{A})$ we set
$$
\MN_k(\mathcal{A})\coloneqq \left\{(m,n)\in \mathds{N}^2 :m,n>1,\ \gcd(m,n)=1\text{ and } (m+n)bk-n(bl-a)=k  \right\},
$$
and we define the set $\MN(\mathcal{A})$, of {\em possible $(m,n)$ for $\mathcal{A}$}, by
$$
\MN(\mathcal{A})\coloneqq \bigcup_{k\in I(\mathcal{A})} \MN_k(\mathcal{A}).
$$
\end{definition}

Next we describe these values as unions of infinite families of $(m,n)$'s, parameterized by $\mathds{N}_0$.

\smallskip

Let $k\in \mathds{N}$ be such that $1\le k<l- \frac{a}{b}$ and set $e_k\coloneqq \gcd(k,bl-a)$. Assume $\gcd\left(b,\frac{bl-a}{e_k}\right)=1$ and let $M_k$ and $N_k$ with $N_k\in \mathds{N}$ minimum satisfying
$$
M_kb-N_k\frac{bl-a}{e_k}=1.
$$
Then
$$
\left\{(M,N)\in \mathds{Z}\times \mathds{N}:Mb-N\frac{bl-a}{e_k}=1\right\} = \left\{\left(M_k+j\frac{bl-a}{e_k},N_k+jb\right): j\in \mathds{N}_k\right\}.
$$
Set
$$
m'_{kj}\coloneqq M_k+j\frac{bl-a}{e_k} - \frac{(N_k+jb)k}{e_k} \qquad\text{and}\qquad n'_{kj}\coloneqq \frac{(N_k+jb)k}{e_k}.
$$
Thus
$$
m'_{kj} = m'_{k0} + j\Delta_k^{\!(1)} \quad\text{and}\quad n'_{kj} = n'_{k0} + j\Delta_k^{\!(2)},\quad\text{where $\Delta_k^{\!(1)}\coloneqq \frac{bl-bk-a}{e_k}$ and $\Delta_k^{\!(2)}\coloneqq \frac{bk}{e_k}$}.
$$
So,
$$
m'_{k,j+1}>m'_{kj}\quad\text{and}\quad n'_{k,j+1}>n'_{kj}\qquad\text{for all $j\in \mathds{N}_0$.}
$$
Hence, by the comments above Definition~\ref{mn families},
we have $1\le m'_{k0},n'_{k0}$. Since we only want consider the $m'_{kj}$'s and $n'_{kj}$'s greater than~$1$, we set
$$
m_{kj} \coloneqq \begin{cases} m'_{kj} & \text{if $n'_{k0}>1$ and $m'_{k0}>1$, }\\ m'_{k,j+1} &\text{otherwise,}\end{cases}\quad\text{and}\quad n_{kj} \coloneqq \begin{cases} n'_{kj} & \text{if $n'_{k0}>1$ and $m'_{k0}>1$, }\\ n'_{k,j+1} & \text{otherwise.}\end{cases}
$$
Clearly
\begin{equation}\label{eq10}
m_{kj} = m_{k0} + j\Delta_k^{\!(1)} \qquad\text{and}\qquad n_{kj} = n_{k0} + j\Delta_k^{\!(2)}.
\end{equation}
With these notations,
$$
\Ss(\mathcal{A},k)\coloneqq \left\{(m,n)\in \mathds{N}^2 : m,n>1\text{ and } (m+n)bk-n(bl-a)=k  \right\}=\{(m_{kj},n_{kj}) : j\in \mathds{N}_0\}.
$$
Since
$$
\MN_k(\mathcal{A}) =  \left\{(m,n)\in \Ss(\mathcal{A},k) : \gcd(m,n)= 1\right\},
$$
we must choose the $(m,n)$'s in $\Ss(\mathcal{A},k)$ such that $\gcd(m,n)=1$. Note that
$$
mb\frac{k}{e_k}+ n\Bigl(b\frac{k}{e_k}-\frac{bl-a}{e_k}\Bigr)=\frac{k}{e_k},
$$
and so $\gcd(m,n)\mid \frac{k}{e_k}$. For  $i\in \bigl\{0,\dots,\frac{k}{e_k}-1\bigr\}$ we define
$$
\MN_{ki}(\mathcal{A})\coloneqq \left\{\left(m_{k,i+j\frac{k}{e_k}},n_{k,i+j\frac{k}{e_k}}\right): j\in \mathds{N}_0 \right\} =  \left\{\left(m_{ki}+j\frac{k}{e_k}\Delta_k^{\!(1)},n_{ki}+j \frac{k}{e_k} \Delta_k^{\!(2)}\right): j\in \mathds{N}_0 \right\}.
$$

\begin{lemma}\label{k sobre e mayor que uno} For all $i\in \bigl\{0,\dots,\frac{k}{e_k}-1\bigr\}$ and all $(m,n)\in \MN_{ki}(\mathcal{A})$, we have
$$
\gcd(m,n)=\gcd(m_{ki},n_{ki}).
$$
Moreover, there exists~$i$ such that $\gcd(m_{ki},n_{ki})=1$.
\end{lemma}

\begin{proof}
Clearly $\MN_{ki}(\mathcal{A})\subseteq \Ss(\mathcal{A},k)$ and so, if $(m,n)\in \MN_{ki}(\mathcal{A})$, then $\gcd(m,n)\mid \frac{k}{e_k}$. Consequently, for $d_{ki}\coloneqq \gcd(m_{ki},n_{ki})$ we have
$$
d_{ki}\mid m_{ki}+j\frac{k}{e_k}\Delta_k^{\!(1)}\quad\text{and}\quad d_{ki}\mid n_{ki}+j\frac{k}{e_k}\Delta_k^{\!(2)}\qquad\text{for all $j$,}
$$
and hence $d_{ki}\mid \gcd(m,n)$ for all $(m,n)\in \MN_{ki}(\mathcal{A})$. Similarly one shows $\gcd(m,n)\mid d_{ki}$, which proves the first assertion. On the other hand, since $\gcd\left(\Delta_k^{\!(1)},\frac{k}{e_k}\right)=1$, the class $\bigl[\Delta_k^{\!(1)}\bigr]$ of~$\Delta_k^{\!(1)}$ in $\mathds{Z}/\frac{k}{e_k}\mathds{Z}$ is invertible, and so
$$
\left\{\left[m_{ki}\right]:i=0,\dots,\frac{k}{e_k}-1\right\} = \frac{\mathds{Z}}{\frac{k}{e_k}\mathds{Z}},
$$
where $[m_{ki}]$ denotes the class of $m_{ki} = m_{k0}+i\Delta_k^{\!(1)}$ in $\mathds{Z}/\frac{k}{e_k}\mathds{Z}$. It follows that there exists an $i$ such that $m_{ki}\equiv 1\pmod{\frac{k}{e_k}}$. Since $d_{ki}\mid m_{ki}$ and $d_{ki}\mid \frac{k}{e_k}$, we obtain $d_{ki}=1$, as desired.
\end{proof}

For each $k\in I(\mathcal{A})$ we let $J_k(\mathcal{A})$ denote $\bigl\{0\le i< \frac{k}{e_k}:\gcd(m_{ki},n_{ki})=1\bigr\}$, where $m_{ki}$ and $n_{ki}$ are as in~\eqref{eq10}. Using the previous results we obtain the following description of the set $\MN(\mathcal{A})$,
$$
\MN(\mathcal{A}) = \bigcup_{k\in I(\mathcal{A})} \MN_k(\mathcal{A})\quad\text{and}\quad \MN_k(\mathcal{A}) = \bigcup_{i\in J_k(\mathcal{A})} \MN_{ki}(\mathcal{A}).
$$

\begin{remark}
Note that for a final corner $\mathcal{A}$ the set $I(\mathcal{A})$ can be empty (for example take $\mathcal{A}=(16\wrs 3,10)$). However, if $k\in I(\mathcal{A})$, then by Lemma~\ref{k sobre e mayor que uno} there exists at least one $(m,n)$-family associated to $\mathcal{A}$. It follows that a final corner $\mathcal{A}=(a\wrs l,b)$ has at least one $(m,n)$-family attached to it, if and only if there exists $k\in\mathds{N}$ with $l-a/b>k\ge 1$, such that
$$
\gcd\left(b,\frac{bl-a}{\gcd(k,bl-a)}\right) = 1.
$$
\end{remark}

In Algorithm~\ref{GetmnFamilies} we obtain the set $\MN(\mathcal{A})$. To achieve this we use the auxiliary function $\BezoutCoefficients(x,y)$ which, for coprime positive integers $x$ and $y$, returns the ordered pair $(M,N)$ of positive integers such that $Mx-Ny=1$ and $N$ is minimal.

\begin{center}
\begin{algorithm}[H]
%  \SetAlgoLined
\DontPrintSemicolon
%   \SetKwFunction{generateCorners}{generateCorners}
%   \SetKwProg{myalg}{Algorithm}{}{}
%   \SetKwProg{myproc}{Procedure}{}{}
%   \myproc{\generateCorners{$\mathcal{A},\mathcal{A}'$}}{
    \KwIn{A final corner $\mathcal{A}=(a\wrs l,b)$.}
	\KwOut{A list $\mnFamilies$ of triples $\bigl((k,i),(m_{ki},n_{ki}),(\Delta^{\!(1)},\Delta^{\!(2)})\bigr)$ such that $k\in I(\mathcal{A})$, $i\in
    J_k(\mathcal{A})$ and $\MN(\mathcal{A})=\bigcup_{k,i} \bigl\{(m_{ki}+j \Delta^{\!(1)},n_{ki}+j \Delta^{\!(2)}):j\in\mathds{N}_0\bigr\}$.}
    {\For {$k=1$ \KwTo $\lceil l-\frac{a}{b}\rceil-1$}{$e \gets \gcd(k,bl-a)$\;
    \If { $\gcd(b,\frac{bl-a}{e})= 1$}{
        $(M,N) \gets \BezoutCoefficients\bigl(b,\frac{bl-a}{e}\bigr)$\;
        $n\gets \frac{Nk}{e}$\;
        $m\gets M-n$\;
        $\Delta^{\!(1)}\gets \frac{bl-a-bk}{e}$\;
        $\Delta^{\!(2)}\gets \frac{bk}{e}$\;
        \If{ $m=1$\normalfont\textbf{ or } $n=1$ }{
            $(m,n)\gets (m,n)+(\Delta^{\!(1)},\Delta^{\!(2)})$ }
            $\ov{k}\gets \frac{k}{e}$\;
            \eIf {$\ov{k}=1$}{
                \bf{add} $\bigl((k,0),(m,n),(\Delta^{\!(1)},\Delta^{\!(2)})\bigr)$ \bf{to} $\mnFamilies$}
                {\For {$i=0$ \KwTo $\ov{k}-1$}{
                $m_i\gets m+i\Delta^{\!(1)}$\\
                $n_i\gets n+i\Delta^{\!(2)}$\\
                \If{$\gcd(m_i,n_i)=1$}{
                 \bf{add} $\bigl((k,i),(m_i,n_i),(\ov{k}\Delta^{\!(1)},\ov{k}\Delta^{\!(1)})\bigr)$ \bf{to} $\mnFamilies$
                    }
					}
				}
                }
    }}
    {\bf RETURN} $\mnFamilies$
     \caption{GetmnFamilies}
     \label{GetmnFamilies}
\end{algorithm}
\end{center}

\section{Program and graphic display}

A website based on these algorithms is under development, making it possible to visualize the construction of chains starting from points below a given upper bound.

\smallskip

The infrastructure for it consists of three parts:

\begin{enumerate}[itemsep=0.5ex,topsep=0.5ex]

\item A C++ implementation of the described pseudocode, along with additional routines to export the information (corners, edges, open and complete chains) to text files formatted for input into an SQL database.

\item An SQL database instance, implemented in PostgreSQL, which organizes the data generated by the C++ program in order to enable easy access by SQL queries.

\item A website mainly developed in the JavaScript language, using the D3.js library for the graphical interface, along with PHP scripts to query the database.

\end{enumerate}
This structure allows a clear separation of responsibilities: the JavaScript code is only concerned with showing the information, assuming it is already suitably formatted, while the C++ program is only concerned with generating the information. It also allows for fast updates to any part of the infrastructure, since each part only depends on the output generated by the others and not on their implementation.
The website consists of a single widget, which contains the following controls:

\begin{enumerate}[itemsep=0.5ex,topsep=0.5ex]

\item An options bar, near the top and below the title. This includes a button to load all points $(x,y)$ with $v_{11}(x,y) < \text{deg}$, for some specified value of $\text{deg}$, and checkboxes for options.

\item A numbered two-dimensional grid, with the ability to zoom and pan, which displays the current items (a collection of corners and edges). A corner $A$ can be clicked to display an edge $(\mathcal{A},\mathcal{A}')$, and the bottom point $A'$ of an edge can be clicked to display the corners generated by it.

\item A collection of \textit{filters} in a right hand panel. These are checkboxes that can be used to only show specific corners. For example, only corners of Type I and Type II, or only corners leading to admissible complete chains.

\end{enumerate}

\section[Admissible complete chains with $v_{11}(A_0)\le 35$]{Admissible complete chains with $\bm{v_{11}(A_0)\le 35}$}

Applying Algorithm~\ref{Main algorithm} with $M\! =\! 35$ we obtain the admissible complete chains $(\mathcal{C}_0,\dots,\mathcal{C}_j,\mathcal{A}_{j+1})$ with $v_{11}(A_0)\le 35$, where $\mathcal{A}_0$ is the first coordinate of $\mathcal{C}_0$. This procedure yields~$14$ admissible complete chains of length~$1$ and~$2$ admissible complete chains of length~$2$. Applying now Algorithm~\ref{GetmnFamilies} with input the final corner $\mathcal{A}_{j+1}$ of any of these chains we obtain the corresponding $(m,n)$-families $\MN_k(\mathcal{A}_{j+1})$ (see Definition~\ref{mn families}). We obtain a two tables. The first consists of~$17$ families of length~$1$, and the second one, of~$7$ families of length~$2$. We only list the cases satisfying equality~\eqref{ecuacion diofantica}. The other cases (satisfying~\eqref{ecuacion diofantica2}) can be obtained by swapping $m$ with $n$.
\begin{center}
\ra{1.2}
\begin{tabular}[t]{ccccccc}
\toprule
\text{Family} & $\mathcal{A}_0$ &$\mathcal{A}_0'$ & $\mathcal{A}_1$& k & m &n \\
\midrule
              $F_1$ & $(4,12)$ & $(1,0)$ & $(7\wrs 4,3)$ & $1$ & $2j+3$ & $3j+4$ \\
              $F_2$ & $(5,20)$ & $(1,0)$ & $(7\wrs 5,2)$ & $1$ & $j+2$ & $2j+3$ \\
              $F_3$ & $(5,20)$ & $(1,0)$ & $(8\wrs 5,3)$ & $1$ & $4j+3$ & $3j+2$ \\
              $F_4$ & $(5,20)$ & $(1,0)$ & $(8\wrs 5,3)$ & $2$ & $2j+3$ & $12j+16$ \\
              $F_5$ & $(5,20)$ & $(1,0)$ & $(9\wrs 5,4)$ & $1$ & $7j+9$ & $4j+5$ \\
              $F_6$ & $(5,20)$ & $(1,0)$ & $(9\wrs 5,4)$ & $2$ & $3j+4$ & $8j+10$ \\
              $F_7$ & $(6,15)$ & $(1,0)$ & $(7\wrs 3,4)$ & $1$ & $j+2$ & $4j+7$ \\
              $F_8$ & $(6,15)$ & $(1,0)$ & $(8\wrs 3,5)$ & $1$ & $2j+3$ & $5j+7$ \\
              $F_9$ & $(7,21)$ & $(1,0)$ & $(11\wrs 7,2)$ & $1$ & $j+2$ & $2j+3$ \\
              $F_{10}$ & $(7,21)$ & $(1,0)$ & $(13\wrs 7,3)$ & $1$ & $5j+7$ & $3j+4$ \\
              $F_{11}$ & $(7,21)$ & $(1,0)$ & $(13\wrs 7,3)$ & $2$ & $j+2$ & $3j+5$ \\
              $F_{12}$ & $(8,24)$ & $(2,0)$ & $(13\wrs 4,5)$ & $1$ & $2j+3$ & $5j+7$ \\
              $F_{13}$ & $(9,21)$ & $(2,0)$ & $(13\wrs 3,7)$ & $1$ & $j+2$ & $7j+13$ \\
              $F_{14}$ & $(9,24)$ & $(1,0)$ & $(7\wrs 3,4)$ & $1$ & $j+2$ & $4j+7$ \\
              $F_{15}$ & $(9,24)$ & $(1,0)$ & $(8\wrs 3,5)$ & $1$ & $2j+3$ & $5j+7$ \\
              $F_{16}$ & $(9,24)$ & $(1,0)$ & $(10\wrs 3,7)$ & $1$ & $4j+3$ & $7j+5$ \\
              $F_{17}$ & $(9,24)$ & $(1,0)$ & $(11\wrs 3,8)$ & $1$ & $5j+2$ & $8j+3$ \\
\bottomrule
\end{tabular}
\end{center}

\end{spacing}

and

\begin{center}
\ra{1.2}
\begin{tabular}[t]{ccccccccc}
\toprule
\text{Family} & $\mathcal{A}_0$ &$\mathcal{A}_0'$ & $\mathcal{A}_1$& $\mathcal{A}_1'$& $\mathcal{A}_2$&k & m &n \\
\midrule
              $F_{18}$ & $(6,18)$ & $(6,15)$ & $(6,15)$ & $(1,0)$ & $(7\wrs 3,4)$ & $1$ & $j+2$ & $4j+7$ \\
              $F_{19}$ & $(6,18)$ & $(6,15)$ & $(6,15)$ & $(1,0)$ & $(8\wrs 3,5)$ & $1$ & $2j+3$ & $5j+7$ \\
              $F_{20}$ & $(6,24)$ & $(6,15)$ & $(6,15)$ & $(1,0)$ & $(7\wrs 3,4)$ & $1$ & $j+2$ & $4j+7$ \\
              $F_{21}$ & $(6,24)$ & $(6,15)$ & $(6,15)$ & $(1,0)$ & $(8\wrs 3,5)$ & $1$ & $2j+3$ & $5j+7$ \\
              $F_{22}$ & $(8,24)$ & $(2,0)$ & $(14\wrs 4,6)$ & $(5\wrs 4,2)$ & $(5\wrs 4,2)$ & $1$ & $j+2$ & $2j+3$ \\
              $F_{23}$ & $(8,24)$ & $(2,0)$ & $(14\wrs 4,6)$ & $(11\wrs 4,4)$ & $(11\wrs 4,4)$ & $1$ & $j+2$ & $4j+7$ \\
              $F_{24}$ & $(8,24)$ & $(2,0)$ & $(14\wrs 4,6)$ & $(5\wrs 4,0)$ & $(19\wrs 8,3)$ & $1$ & $2j+3$ & $3j+4$ \\
\bottomrule
\end{tabular}
\end{center}

\smallskip

For each one of these chains let $(a\wrs l,b)$ be its final corner and let $e_k = \gcd(k,bl-a)$. In all the cases except $F_4$, we have $k/e_k=1$. In case $F_4$ we have $k/e_k = 2$ and $J_k(8\wrs 5,3) = \{1\}$.

\smallskip

We claim that the families $F_{18}$, $F_{19}$, $F_{20}$ and $F_{21}$ can not be obtained from a standard $(m,n)$-pair $(P,Q)$ as in Theorem~\ref{standard pair generates complete chain}. Note that with the notations used in that theorem for the four families we have
$$
(\rho_0,\sigma_0) = \dir(A_0-A'_0)=(1,0)\qquad\text{and}\qquad (\rho_1,\sigma_1) = \dir(A_1-A'_1)=(3,-1).
$$
Hence, by the second equality in~\eqref{q_k} we have $q_1 = 3$. If there were an $(m,n)$-pair $(P,Q)$ for one the families, then by equality~\eqref{eq1} and Remark~\ref{admis} with $h=0$ and $i=k=1$ there exists $R\in L$ such that $\ell_{10}(P)=R^{3m}$. Let $(a,b) = A_0$ and $(a',b') = A'_0$. Since
$$
\ell_{10}(P) = x^{a'm}y^{b'm}p(y)\qquad\text{where $p(0)\ne 0$ and $\deg(p) = mb-mb'$,}
$$
in the first two cases there exist $\lambda_P,\lambda\in K^{\times}$ such that
$$
\ell_{10}(P)= \lambda_p(x^2y^5(y-\lambda))^{3m},
$$
while in the last two cases there exist $\lambda_P,\lambda,\lambda',\lambda''\in K^{\times}$ such that
$$
\ell_{10}(P)= \lambda_p(x^2y^5(y-\lambda)(y-\lambda')(y-\lambda''))^{3m}\quad\text{and}\quad\text{$\lambda\notin\{\lambda',\lambda''\}$ or $\lambda = \lambda' = \lambda''$}.
$$
Define $\varphi\in \Aut(L)$ by
$$
\varphi(x)\coloneqq x\quad\text{and}\quad \varphi(y)\coloneqq y+\lambda.
$$
By~\cite{GGV1}*{Proposition~3.9} we know that, for all $H\in L$,
\begin{align*}
&\ell_{10}(\varphi(H)) = \varphi(\ell_{10}(H)),\qquad \en_{10}(\varphi(H)) = \en_{10}(H)\notag
\shortintertext{and}
&\ell_{\rho_1,\sigma_1}(\varphi(H)) = \ell_{\rho_1,\sigma_1}(H)\quad\text{for all $(1,0)< (\rho_1,\sigma_1) < (-1,0)$}.%\label{eq11}
\end{align*}
Using this with $H =P$ and $H = Q$, we obtain that
$$
\frac{v_{11}(\varphi(P))}{v_{11}(\varphi(Q))} = \frac{v_{10}(\varphi(P))}{v_{10}(\varphi(Q))} = \frac{m}{n}\quad\text{and}\quad v_{1,-1}(\en_{10}(\varphi(P))) < 0.
$$
Hence $(\varphi(P),\varphi(Q))$ is an $(m,n)$-pair, since, by~\cite{GGV1}*{Proposition~3.10},
$$
[\varphi(P),\varphi(Q)] = [P,Q]\in K^{\times}.
$$
Moreover
$$
\ell_{10}(\varphi(P)) = \varphi(\ell_{10}(P)) = \lambda_p(x^2(y+\lambda)^5y)^{3m} = \lambda_p x^{6m}y^{3m} (y+\lambda)^{15m}
$$
in the first two cases, and
$$
\ell_{10}(\varphi(P)) = \varphi(\ell_{10}(P)) = \lambda_p x^{6m}y^{3m}(y+\lambda-\lambda')^{3m}(y+\lambda-\lambda'')^{3m} (y+\lambda)^{15m}
$$
in the last two cases. So, in the first two cases
$$
\frac{1}{m}\st_{10}(\varphi(P)) = (6,3),
$$
and the same occurs in the last two cases if $\lambda\notin\{\lambda',\lambda''\}$. Hence, by~\cite{GGV2}*{Remark~3.2} the point $(6,3)$ is a last lower corner. But this is impossible by~\cite{GGV2}*{Remark~3.29}. On the other hand if in the last two cases $\lambda = \lambda' = \lambda''$, then
$$
\frac{1}{m}\st_{10}(\varphi(P)) = (6,9),
$$
and so $(\varphi(P),\varphi(Q))$ is a standard $(m,n)$-pair. Let $(A,A',(\rho,\sigma))$ be the starting triple of $(\varphi(P),\varphi(Q))$. Since
$$
(1,-1)<(\rho,\sigma)\le \Pred_{\varphi(P)}(1,0),
$$
arguing as in the proof of \cite{GGV1}*{Proposition~6.1(9)} we obtain that
$$
v_{11}(A)\le v_{11}(6,9) = 15.
$$
But this is impossible by \cite{GGV1}*{Proposition~6.5}.

\begin{remark}
The possible counterexample in $F_{13}$ with $j=1$ was analyzed extensively by Orevkov in~\cite{O} (see~\cite{O}*{Lemma~4.1(a)}).
\end{remark}

\section[Possible counterexamples with $\max(\deg(P),\deg(Q))\le 150$]{Possible counterexamples with $\bm{\max(\deg(P),\deg(Q))\le 150}$}

In~\cite{M} there are listed four cases (which correspond to six cases in our terminology) of possible counterexamples with $\max(\deg(P),\deg(Q))\le 100$. They are discarded by hand. Here we describe the shape of the $34$ possible counterexamples with $\max(\deg(P),\deg(Q))\le 150$.  We only list the cases satisfying equality~\eqref{ecuacion diofantica}. The other cases (satisfying~\eqref{ecuacion diofantica2}) can be obtained by swapping $m$ with $n$. Thirteen of them correspond to a choice of $(m,n)$ in some of the families listed in the previous section, as can be seen in the following table, where the red pairs correspond to possible counterexamples with $\max(\deg(P),\deg(Q))\le 100$.

\begin{center}
\ra{1.2}
\begin{tabular}[t]{clc}
\toprule
\text{Family} & $( m,n)$& $\max\{\deg(P),\deg(Q)\}$ \\
\midrule
              $F_{1}$&{\color{red}(3,4)}&64 \\
              $F_{1}$&(5,7)&112 \\
              $F_{2}$&{\color{red}(2,3)}&75 \\
              $F_{2}$& (3,5)&125 \\
              $F_{3}$&{\color{red}(3,2)}&75 \\
              $F_{7}$&(2,7) & 147\\
              $F_{8}$&(3,7)&147 \\
              $F_{9}$&{\color{red}(2,3)} & 84\\
              $F_{9}$& (3,5) & 140\\
              $F_{11}$&(2,5) & 140 \\
              $F_{17}$&{\color{red}(2,3)}&99 \\
              $F_{22}$& {\color{red}(2,3)*}& 96 \\
              $F_{24}$&(3,4) &128\\
\bottomrule
\end{tabular}
\end{center}
Five of them correspond to the six cases found by Moh, one of the cases of Moh was discarded by the algorithm because it featured $(A_0,A_0')=((7,21),(2,1))$, and $(2,1)\notin \PLLC$. The sixth red case, marked with a star, corresponds to $F_{22}$. This case was probably discarded as a possible counterexample by Heitmann (with no mention to it) by symmetry reasons. This case corresponds to the first case listed in~\cite{LCW}*{pag. 426} with $\delta_3=1/4$, $\delta_2=9/16$ and $\delta_1=7/12$. In Proposition~\ref{caso antisimetrico} we show that we can discard it.

There are $9$ other possible pairs with a complete chain of length~$1$, which we list in the following table:
\begin{center}
\ra{1.2}
\begin{tabular}[t]{cclc}
\toprule
$A_0$ & $A_1$& $(m,n)$ & $\max\{\deg(P),\deg(Q)\}$ \\
\midrule
             (7,35)& (19/7,5)& (2,3)& 126\\
              (7,42)&(13/7,6) &  (3,2)& 147\\
              (7,42)&(13/7,6) &   (2,3)& 147\\
              (8,28)&(7/4,3) & 	 (3,4) & 144\\
              (8,28)&(11/4,7) & 	 (3,2)& 108\\
              (9,36)&(17/9,4) & 	 (3,2)& 135\\
              (9,36)&(17/9,4) & 	 (2,3)& 135\\
              (11,33)&	(19/4,8) & 	 (2,3)&132\\
              (12,33)&(11/3,8) & 	 (2,3)&135\\
\bottomrule
\end{tabular}
\end{center}

There are also $11$ other possible pairs with a complete chain of length~$2$, which we list in the following table:
\begin{center}
\ra{1.2}
\begin{tabular}[t]{ccclc}
\toprule
$A_0$ & $A_1$&$A_2$& $(m,n)$ & $\max\{\deg(P),\deg(Q)\}$ \\
\midrule
              (8,32)&(8,28) & (11/4,7) & 	 (3,2)&120 \\
               (8,40)&(8,28) &(11/4,7) & 	 (3,2)&144 \\
              (9,27) & (9,24) & (11/3,8)&	 (2,3)&108\\
              (9,36) & (9,24) & (11/3,8) &   (2,3)&135 \\
              (10,40) & (16/5,6) & (23/10,3) & 	 (3,2)& 150\\
               (10,40)& (18/5,8) &(8/5,3) &	 (3,2)&150\\
             (12,30)&(16/3,10) & (11/6,3) &  	 (3,2)&126 \\
             (12,36) & (12,33) & (11/3,8) &	 (2,3)&144\\
             (12,36) & (9,24) & (11/3,8) &	 (2,3)&144\\
(12,36)&(21/4,9) & (19/4,8) & 	 (2,3)&144\\
 (12,36)&(21/4,9) &(12/4,5) & 	 (2,3)&144\\
\bottomrule
\end{tabular}
\end{center}

Finally there is another possible pair with a complete chain of length~$3$:
\begin{center}
\begin{tabular}[t]{cccccc}
\toprule
 $A_0$ & $A_1$&$A_2$& $A_3$ & $(m,n)$ & $\max\{\deg(P),\deg(Q)\}$ \\
\midrule
              (12,36)&(12,30) & (16/3,10) & (11/6,3)&	 (3,2)&144 \\
\bottomrule
\end{tabular}
\end{center}

\begin{proposition}\label{caso antisimetrico} The example corresponding to $F_{22}$ with $(m,n)=(2,3)$ can not be obtained from a standard $(m,n)$-pair $(P,Q)$ as in Theorem~\ref{standard pair generates complete chain}.
\end{proposition}

\begin{proof} With the notations used in Theorem~\ref{standard pair generates complete chain}, we have
$$
\mathcal{A}_1=(14\wrs 4,6),\qquad \mathcal{A}_1'=\mathcal{A}_2=(5\wrs 4,2) \qquad\text{and}\qquad (\rho_1,\sigma_1) = \dir(A_1-A'_1)=(16,-9).
$$
Consequently,
$$
\ell_{16,-9}(P_1) = x^{\frac{5m}{4}}y^{2m} p(z)\quad\text{with $z\coloneqq x^{\frac{9}{16}}y$, $p\in K[z]$ and $p(0)\ne 0$.}
$$
Combining this with equality~\eqref{eq1} and the fact that $\gap(16,4) = 4$ we obtain that
$$
\ell_{16,-9}(P_1) = \lambda_p x^{\frac{5m}{4}}y^{2m} (z^4-\lambda')^m\qquad\text{where $\lambda',\lambda_p\in K^{\times}$.}
$$
Hence
$$
\ell_{16,-9}(P_1) = \lambda_p x^{\frac{5m}{4}}y^{2m} (z^4-\lambda^4)^m = \lambda_p x^{\frac{5m}{4}}y^{2m} (z-\lambda)^m(z^3+z^2\lambda+z\lambda^2+\lambda^3)^m
$$
where $\lambda\in K^{\times}$ is such that $\lambda^4=\lambda'$. Thus the multiplicity $m_{\lambda}$ of $\lambda$ as a root of $p(z)$ equals $m$. Define $\varphi\in \Aut(L^{(16)})$ by $\varphi(x)\coloneqq x$ and $\varphi(y)\coloneqq  y+\lambda x^{-9/16}$. By~\cite{GGV1}*{Proposition~3.9} we know that,
\begin{align*}
&\ell_{16,-9}(\varphi(H)) = \varphi(\ell_{16,-9}(H)),\qquad \en_{16,-9}(\varphi(H)) = \en_{16,-9}(H)\notag
\shortintertext{and}
&\ell_{\rho_1,\sigma_1}(\varphi(H)) = \ell_{\rho_1,\sigma_1}(H)\quad\text{for all $(16,-9)< (\rho_1,\sigma_1) < (-16,9)$},
\end{align*}
for all $H\in L^{(16)}$. Using this with $H =P_1$ and $H = Q_1$, we obtain that
$$
\frac{v_{11}(\varphi(P_1))}{v_{11}(\varphi(Q_1))} = \frac{v_{10}(\varphi(P_1))}{v_{10}(\varphi(Q_1))} = \frac{m}{n}\quad\text{and}\quad v_{1,-1}(\en_{16,-9}(\varphi(P_1))) < 0.
$$
Hence $(\varphi(P_1),\varphi(Q_1))$ is an $(m,n)$-pair, since $[\varphi(P_1),\varphi(Q_1)] = [P_1,Q_1]\in K^{\times}$, by~\cite{GGV1}*{Proposition~3.10}. Moreover
\begin{align*}
\ell_{16,-9}(\varphi(P_1)) & = \varphi(\ell_{16,-9}(P_1)) \\
& = \lambda_p x^{\frac{5m}{4}} (y+\lambda x^{\frac{-9}{16}})^{2m}((z+\lambda)^4-\lambda^4))^m\\
%
%& = \lambda_p x^{\frac{5m}{4}}  x^{\frac{-18m}{16}} (z+\lambda)^{2m} ((z+\lambda)^4-\lambda^4))^m\\
%
%&= \lambda_p x^{\frac{5m}{4}}  x^{\frac{-18m}{16}} z (z+\lambda)^{2m} (z^3 + 4z^2\lambda + 6z\lambda^2 + 4\lambda^3)^m\\
%
&= \lambda_p x^{\frac{11m}{16}}y^{m}(z+\lambda)^{2m} (z^3 + 4z^2\lambda + 6z\lambda^2 + 4\lambda^3)^m,
\end{align*}
and so $\bigl(\frac{11}{16},1\bigr) = \frac{1}{m}\st_{16,-9}(\varphi(P_1))$. Now note that the inequality~(5.9) in \cite{GGV1}*{Proposition~5.18} is satisfied for $a=20$, $b=6$, $l=16$, $\rho=16$ and $\sigma=-9$. Consequently, by that proposition, the $(m,n)$-pair $(\varphi(P_1),\varphi(Q_1))$ has a regular corner at $(11/16,1)$. Since $\gcd(11,1) = 1$, by \cite{GGV1}*{Proposition~5.19} there exists a (possibly different) $(m,n)$-pair $(P',Q')$ in $L^{(16)}$ such that $(11/16,1)$ is the first entry of a regular corner of type~I of $(P',Q')$. By Proposition~3.1 we can assume that $(11/16,1)$ is the first entry of a regular corner of type~I.b) of $(P',Q')$. Then $a=11$, $b=1$, $l=16$, $k\in\{1,2,3,4\}$, $e_k=1$ and $\{m,n\}=\{2,3\}$ in the setting of Proposition~\ref{extremosfinales1}. Hence
$$
1=(m+n)b-\frac{me_k}{k} \frac{bl-a}{e_k}=5-\frac{m}{k}5=5\frac{k-m}{k}
$$
or
$$
1=(m+n)b-\frac{ne_k}{k} \frac{bl-a}{e_k}=5-\frac{n}{k}5=5\frac{k-n}{k}.
$$
But both equalities are evidently false for $n,m\in\{2,3\}$ and $k\in\{1,2,3,4\}$, since $5\nmid k$.
\end{proof}

%\end{spacing}

\end{document}